\documentclass[a4paper]{amsart}
\usepackage[utf8]{inputenc}
\usepackage{amsmath,amssymb,amsthm}
\usepackage[british]{babel}
\usepackage{eucal}
\usepackage{bm}
\usepackage{graphicx}
\usepackage[dvipsnames]{xcolor}
\usepackage{url}
\usepackage{tikz}
\usepackage{pgfplots}
\usetikzlibrary{pgfplots.groupplots}
\pgfplotsset{compat=1.3}
\newcommand{\LLL}{\lvert\!\lvert\!\lvert}
\newcommand{\RRR}{\rvert\!\rvert\!\rvert}
 \newtheorem{theorem}{Theorem}
 
 \newtheorem{lemma}{Lemma}
 \theoremstyle{definition}
  \newtheorem{definition}[lemma]{Definition}
\numberwithin{equation}{section}
\numberwithin{lemma}{section}

\begin{document}
\title[Adini FEM with hanging nodes]{The Adini finite element \\ on locally refined meshes}
\author{D. Gallistl}
\thanks{Supported by the European Research Council
 (StG \emph{DAFNE}, ID 891734).}
\address{Institut f\"ur Mathematik, Universit\"at Jena, 07743 Jena, Germany}
\email{dietmar.gallistl [at] uni-jena.de}
\date{}

\begin{abstract}
This work introduces a locally refined version of the Adini finite element
for the planar biharmonic equation
on rectangular partitions with at most one hanging node per edge.
If global continuity of the discrete functions is enforced,
for such method there is some freedom in assigning the normal
derivative degree of freedom at the hanging nodes.
It is proven that the convergence order $h^2$ known for regular
solutions and regular partitions is lost for any such choice,
and that assigning an average of the normal derivatives at the
neighbouring regular vertices is the only choice that achieves
a superlinear order, namely $h^{3/2}$ on uniformly refined meshes.
On adaptive meshes,
the method behaves like a first-order scheme.
Furthermore, the reliability and efficiency of an explicit
residual-based error estimator are shown up to the best approximation
of the Hessian by certain piecewise polynomial functions.
\end{abstract}


\maketitle

%
%
\section{Introduction and main results}
While Galerkin methods enjoy the error bound from C\'ea's lemma and, therefore,
local mesh refinement with nested spaces does not increase the approximation
error, in nonconforming discretizations ---a popular choice for the
biharmonic equation--- local refinement of the mesh resolution may
potentially disimprove the situation.
The main purpose of this work is an analysis of this phenomenon
in a model situation.
The Adini finite element method (FEM) is one of the earliest
methods for numerically solving the biharmonic equation
\cite{AdiniClough1961,Ciarlet1978}.
It is a standard four-noded rectangular element in the engineering
literature, and therein also referred to as
Adini--Clough--Melosh element \cite{Onate2009}.
Given a rectangular partition $\mathcal T$ of the underlying domain
$\Omega\subseteq\mathbb R^2$, the shape function space for
every rectangle $T$ is the space of cubic polynomials over $T$
enriched by the two monomials $x^3 y$ and $x y^3$, where
the Cartesian coordinates of a point in the plane are denoted by
$x,y$ and the mesh is assumed to be aligned with the Cartesian
axes. The corresponding twelve degrees of freedom are the point
evaluation of a function and the evaluation of its first-order partial
derivatives in any of the four vertices.
The resulting finite element, schematically shown in the mnemonic
diagram of Figure~\ref{f:adiniFE}, is easy to implement and
its a~priori error analysis is
theoretically well understood when regular partitions are used.
Regularity of a partition $\mathcal T$
means that if any vertex $z$ of an element $T\in\mathcal T$
belongs to some element $K\in\mathcal T$, it is automatically also
a vertex of $K$. For such regular meshes it is known that the
method converges at the order $h^2$ under uniform mesh refinement
if the solution is sufficiently regular, $h$ being the maximal
mesh size
\cite{Ciarlet1978,LascauxLesaint1975,HuYangZhang2016}.
In presence of singularities of the solution, the convergence order
is significantly reduced and adaptive mesh refinement towards the
singularity becomes mandatory,
a case not studied so far in the literature on the Adini FEM.
On rectangular partitions with bounded
aspect ratio, such refinement
necessarily requires elements with irregular vertices
(commonly called hanging nodes), i.e., a vertex $z$ of a rectangle
$T$ may belong to an edge of another rectangle $K$ without being a
vertex of it. The degrees of freedom attached to that hanging node
are then subject to some interpolation constraint.
The typical situation is displayed in Figure~\ref{f:adiniFE}.
In the case of the Adini element,
the value and the derivative in the direction tangential to the
edge are prescribed by the condition of the function to be globally
continuous. The continuity condition for the partial derivative in
the direction normal to the edge, however, is not canonically prescribed
because the Adini FEM is a nonconforming method, meaning that the discrete
functions are globally continuous but their gradients may be discontinuous
so that the discrete functions may possibly not belong to $H^2(\Omega)$,
the energy space for the biharmonic equation.
Two obvious possibilities (out of many others) are:
either the degree of freedom is set in such a way that it
interpolates the partial derivative on the neighbouring element;
or it is simply chosen as the 
average by linear interpolation
of the partial derivatives at the
neighbouring vertices determining the edge that contains the irregular vertex
$z$ in its interior.
It is obvious that the latter choice cannot retain the quadratic approximation
order $h^2$ known from the regular case because the averaging operation does not
conserve cubic polynomials.
However, in this work it is proven that it is the only possible choice
(in the class of linear, local,
and scaling-invariant couplings) that yields
a superlinear order, namely $h^{3/2}$ on uniform refinements of an initial
irregular mesh subjected to the condition of Definition~\ref{d:meshcondition}
below.

\begin{figure}
 \begin{center}
 \begin{tikzpicture}[scale=2]
  \draw (0,0)--(1.5,0)--(1.5,1)--(0,1)--cycle;
  \foreach \x/\y  in {0/0,1.5/0,1.5/1,0/1}
      {  \fill (\x,\y) circle (1pt);
         \draw (\x,\y) circle (1.7pt);
      }
\end{tikzpicture}
\qquad
 \begin{tikzpicture}[scale=2]
  \draw (0,0)--(1.5,0)--(1.5,1)--(0,1)--cycle;
  \draw (1.5,0)--(2.5,0);
  \draw (1.5,1)--(2.5,1);
  \draw (1.5,.5)--(2.5,.5);
  \fill (1.5,.5) circle (1pt);
  \draw[ultra thick,->] (1.5,.5)--(1.5,.7);
  \draw[thick,->] (1.5,.5)--(1.3,.5);
  \foreach \x/\y  in {0/0,1.5/0,1.5/1,0/1}
      {  \fill (\x,\y) circle (1pt);
         \draw (\x,\y) circle (1.7pt);
      }
\end{tikzpicture}
 \end{center}
 \caption{Mnemonic diagram of Adini's finite element (left);
          degrees of freedom at a hanging node (right).}
 \label{f:adiniFE}
\end{figure}
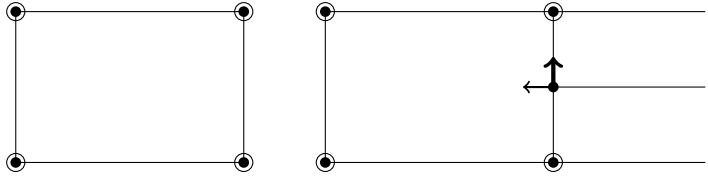

The design of the Adini element does not involve average integrals of normal
derivatives over edges as degrees of freedom, in contrast to
nonconforming methods like the Morley element and others
\cite{LascauxLesaint1975}. This prevents the element from passing
certain patch tests, and the error analysis is more involved and relies
on the choice of the shape function space, which is the same as for the
lowest-order Serendipity element \cite{BrennerScott2008}.
Consequently, a reliability proof for a residual-based error estimator has
not been available \cite{CarstensenGallistlHu2013,GallistlTian2024}.
Furthermore, the definition of the element on meshes with hanging
nodes is not straightforward because an analogue to
\cite[condition (A2)]{CarstensenHu2009} is not satisfied by
the normal derivative.
As the first main result in this work, it is shown that
the quadratic approximation order is necessarily lost in the presence
of hanging nodes, showing that best-approximation results in the
fashion of \cite{Gudi2010} are unavailable.
It is shown that a suitable assignment of local degrees of freedom
at hanging nodes can lead to $h^{3/2}$ convergence.

\begin{theorem}[a priori error estimate]\label{t:pri}
 Let $f\in L^2(\Omega)$ be such that the exact solution
 $u$ to the biharmonic problem
 \eqref{e:bih}
 satisfies
 $u\in H^4(\Omega)\cap W^{3,\infty}(\Omega)$. Let $(\mathcal T_h)_h$ be a sequence of uniform
 refinements of an initial partition
 that satisfies the mesh condition of Definition~\ref{d:meshcondition}
 and contains at least one
 irregular vertex.
 Let $u_h\in V_h$ denote the finite element
 solution to \eqref{e:bihd}
 where $V_h$ is the Adini finite element space with regular
 assignment in the sense of Definition~\ref{d:admissible}.
 The averaging assignment, that is the choice of $V_h$
 according to \eqref{e:Vh_av},
 satisfies
 $$
   \LLL u-u_h \RRR_h\lesssim h^2\|u\|_{H^4(\Omega)}
          +h    \|u\|_{H^3(\cup \mathcal T^{\mathrm{irr}})}
 ,
 $$
where $\cup \mathcal T^{\mathrm{irr}}$ from \eqref{e:Tirrdef}
is the area covered by elements
with irregular vertices.
 In particular, it satisfies the asymptotic bound
 $$
   \LLL u-u_h \RRR_h\lesssim h^{3/2} (\|u\|_{H^4(\Omega)}+\|u\|_{W^{3,\infty}(\Omega)})
 $$
 on uniformly refined meshes.
 For any other admissible assignment there exists a right-hand side $f$
 such that the solution $u\in C^\infty(\overline\Omega)$
 is smooth, but
 $$
 \LLL u-u_h \RRR_h \gtrsim h .
 $$
\end{theorem}

Furthermore, a residual-based a~posteriori error estimator is shown
to be reliable and efficient up to terms that are second-order accurate
on regular meshes,
but only first-order on more general meshes
(details on the notation follow in \S\ref{s:prelim}).

\begin{theorem}[a posteriori error estimate]\label{t:post}
Let $\mathcal T$ be a partition satisfying
the mesh condition of Definition~\ref{d:meshcondition}
and $V_h$ be
chosen according to the averaging assignment \eqref{e:Vh_av}.
The solution $u$ to the biharmonic problem \eqref{e:bih}
with right-hand side $f\in L^2(\Omega)$
and its Adini
finite element discretization $u_h$ from \eqref{e:bihd} satisfy,
with $\bm\eta$, $\bm\eta(T)$ defined in \eqref{e:estimator},
the reliability estimate
 $$
 \LLL u-u_h \RRR_h
 \lesssim \bm\eta
 $$
 and local efficiency
 $$
   \bm\eta(T)
   \lesssim
   \LLL u-u_h \RRR_{h,\omega_T}
   + \|(1-\Pi^{\mathcal T})D^2u\|_T
   + \|h^2 (1-\Pi_0)f\|_{\omega_T}
 $$
 for any $T\in\mathcal T$ with element patch $\omega_T$
 and the projection $\Pi^{\mathcal T}$
 from \eqref{e:PiTproj}.
\end{theorem}

While its efficiency part is not new and can be proven with standard
arguments \cite{Verfuerth2013},
more importantly
Theorem~\ref{t:post} also provides a reliability result of an a~posteriori
error estimator for the Adini element,
which partly proves a conjecture of
\cite{CarstensenGallistlHu2013} and explains the results of
their numerical experiments.
Therein, the error estimator $\bm\eta$ (up to the additional
local projection error $\|(1-\Pi^{\mathcal T})D^2u\|_T$
not considered there)
was experimentally observed to be an upper error bound
on uniformly refined meshes.
Theorem~\ref{t:post} theoretically justifies the observed convergence rates of
the error estimator in \cite{CarstensenGallistlHu2013}.

The results presented here allow for two conclusions.
The first one is that the Adini FEM can be used as a first-order
method for resolving corner singularities
and hints potential for resolving non-rectilinear (possibly curved) domains.
Since the Adini shape function space is that of the Serendipity family
\cite{BrennerScott2008}, the element cannot be mapped to general quadrilaterals
like trapeziums without loss of approximation quality, see the discussion
in \cite{LascauxLesaint1975}.
The local resolution variant of the method proposed here thus makes the Adini
FEM more competitive for such situations.
In some cases, it even satisfies superlinear convergence.
Secondly, and perhaps more fundamentally, the analysis shows that the quadratic
convergence order is necessarily lost under fairly reasonable coupling conditions
at hanging nodes. This highlights that nonconforming methods do not naturally
generalize to irregular partitions in absence of further structural conditions.
In particular, local refinement can significantly deteriorate the approximation
(as proven in Theorem~\ref{t:pri} and illustrated by numerical results in
\S\ref{ss:result_unif}),
and
best-approximation results analogous to those formulated in \cite{Gudi2010}
do not hold in this case.

This article is organized as follows:
\S\ref{s:prelim} defines the necessary data structures around
finite element meshes and introduces the Adini element.
The assignment at hanging nodes is discussed in
\S\ref{s:hanging}.
The proof of Theorem~\ref{t:pri} is provided in \S\ref{s:pri},
while  \S\ref{s:post} provides the proof of
Theorem~\ref{t:post}.
Comments on the extension to more general boundary conditions
follow in \S\ref{s:bc}.
Numerical experiments are shown in
\S\ref{s:num}.
Finally,
some important but technical estimates for discrete functions are provided in
the Appendices \S\ref{app:adinibasis}--\S\ref{app:qiAdini}.

\medskip
Throughout this work, standard notation on Lebesgue and Sobolev spaces
is used. The $L^2$ norm over a measurable set $\omega$ is denoted by
$\|\cdot\|_\omega$ with the convention $\|\cdot\|=\|\cdot\|_\Omega$.
Polynomial functions of total resp.\ partial degree not greater than
$k$ are denoted by $P_k$ resp.\ $Q_k$.
The notation $a\lesssim b$ or $b\gtrsim a$
indicates an inequality $a\leq Cb$ with a constant
independent of the mesh size; $a\approx b$ means $a\lesssim b\lesssim a$.

\section{Adini's finite element for the biharmonic equation}
\label{s:prelim}

Let $\Omega\subseteq \mathbb R^2$ be an open and bounded
rectilinear Lipschitz polygon.
Given a right-hand side $f\in L^2(\Omega)$,
the biharmonic problem with clamped boundary conditions seeks
$u\in H^2_0(\Omega)$ such that
\begin{equation}\label{e:bih}
  a(u,v) = (f,v)_{L^2(\Omega)} \quad\text{for all }v\in H^2_0(\Omega),
\end{equation}
where the bilinear form $a$ is defined by
$$
 a(v,w) := \int_\Omega D^2 v: D^2w \quad\text{for any }v,w\in H^2(\Omega)
$$
and the colon $:$ denotes the Frobenius inner product of matrices.

The following notation related to a partition $\mathcal T$
of $\Omega$ is used.
The set of vertices (extremal points) of a rectangle is denoted by
$\mathcal V(T)$.
The set of all vertices of $\mathcal T$ is denoted by $\mathcal V$.
A vertex $z\in \mathcal V$ for which $z\in T\in\mathcal T$ implies
$z\in\mathcal V(T)$, i.e., $z$
is one of the four vertices of $T$, is called a regular vertex,
and the set of such vertices is denoted by
$\mathcal V^{\mathrm{reg}}$.
The remaining irregular vertices are denoted by
$\mathcal V^{\mathrm{irr}}=\mathcal V\setminus\mathcal V^{\mathrm{reg}}$.
Throughout this work, the notions \emph{hanging node} and
\emph{irregular vertex} are used interchangeably.
Any irregular $z\in \mathcal V^{\mathrm{irr}}$ necessarily lies on the
interior of an edge $E$ of some rectangle $T$ that is the convex
hull of two vertices $z_1,z_2\in\mathcal V(T)$,
called the neighbouring vertices.
In particular $z\in E=\operatorname{conv}\{z_1,z_2\}$.
Throughout this paper, we work on classes of partitions with
uniformly bounded aspect ratio.
The $L^2$ projection to piecewise (possibly discontinuous) $P_k$ functions
is denoted by $\Pi_k$.
For $z\in\mathcal V$ and $T\in\mathcal T$ we define the usual patches
$$
\omega_z:=\operatorname{int} (\cup\{K\in\mathcal T: z\in K\})
\quad
\text{and}
\quad
\omega_T:=\cup \{\omega_z : z\in\mathcal V(T)\}.
$$
The outer unit normal of the boundary of
a rectangle $T$ is denoted by $n_T$.
The set of all edges is denoted by $\mathcal E$.
Every edge has a (globally fixed) normal vector $n_E$
and a tangential vector $t_E$.
If the meaning is clear from the context and there is no risk
of confusion, the symbols $n$ and $t$ are sometimes used without index
in expressions like $\partial^2_{nn}$, $\partial^2_{nt}$, etc.
The diameter of a rectangle $T$ and an edge $E$ are denoted by
$h_T$ and $h_E$, respectively. The piecewise constant mesh-size function
$h$ is defined by $h|_T:=h_T$ for any $T\in\mathcal T$.
If the letter $h$ is used in global expressions like $O(h^s)$
or outside norms, it denotes the maximum of the mesh size
function.

The piecewise Hessian with respect to $\mathcal T$ is denoted by
$D^2_h$, and the index $h$ is also used to indicate piecewise partial
derivatives $\partial_{j,h}$ of piecewise smooth functions.
Any rectangle $T\subseteq\mathbb R^2$
will be assumed to be aligned with the Cartesian axes, so that
any of its faces is parallel to either the $x$ or $y$ axis.
The shape function space $\mathcal A$
is that of cubic polynomials
enriched by the two elements $xy^3$ and $x^3y$, written
$$
 \mathcal A = P_3 + \langle xy^3,x^3y\rangle
 ,
$$
where angle brackets denote the linear hull.
If there is no risk of confusion,
a polynomial function will not be
distinguished from its restriction to or its extension
from some subdomain of $\mathbb R^2$
throughout this work.
Given a rectangle $T$, the twelve
degrees of freedom of the Adini finite element
are the point evaluations of a function and of
its first partial derivatives in those vertices.
A corresponding diagram is displayed in Figure~\ref{f:adiniFE}.
Given $\Omega$,
let $\mathcal T$ be a finite partition into rectangles such
that the elements of $T$ cover the domain
$\cup_{T\in\mathcal T} T = \overline \Omega$
and the intersection of the interior of any two distinct elements is empty.
The space of piecewise Adini functions reads
$$
\mathcal A(\mathcal T):=
\{ v\in L^\infty(\Omega) : v|_T\in\mathcal A \text{ for any }T\in\mathcal T\}.
$$
If $\mathcal T$ is any such partition (with or without hanging nodes),
the global finite element space with clamped boundary condition
and gradient continuity at the regular vertices reads
$$
\widehat V_h
:=
C(\overline\Omega)\cap\left\{v\in\mathcal A(\mathcal T) \left|
                \begin{aligned}
                                  \nabla v \text{ is continuous in the interior
                                  regular vertices of } \mathcal T
                                  \\
                               v \text{ and } \nabla v
                                  \text{ vanish on the boundary vertices of }
                                  \mathcal T
                \end{aligned}
        \right.
       \right\}.
$$
The continuity requirement shows that $\widehat V_h$
is spanned by $\mathcal A(\mathcal T)$ functions that
are continuous in all vertices, with continuous gradient
in all regular vertices and with continuous tangential derivative
at irregular vertices (`tangential' referring to the edge containing the
hanging node). No condition is made on the normal derivative at
such vertex although it is a local degree of freedom for the
finite element.
For regular partitions, $V_h=\widehat V_h$ is the standard
Adini finite element space known from the literature.
In this case it is known that $V_h\subseteq C(\overline\Omega)$ is a space of
continuous functions with possibly discontinuous piecewise derivatives.
This means that $V_h$ is a subspace of the Sobolev space $H^1_0(\Omega)$
but in general not a subspace of the energy space $H^2_0(\Omega)$
for the biharmonic problem, whence it is referred to as nonconforming.
If the partition contains irregular vertices, a subspace
$V_h\subseteq \widehat V_h$ needs to be considered such that the
discrete problem is well posed.
The Adini finite element discretization is based on the discrete
bilinear form
$$
 a_h(v,w) := \int_\Omega D^2_h v: D^2_hw
 \quad\text{for any }v,w\in H^2_0(\Omega)+\widehat V_h
 ,
$$
where $D^2_h$ denotes the piecewise Hessian with respect to $\mathcal T$.
Under the admissibility condition of Definition~\ref{d:admissible} below,
$V_h$ is such that $a_h$ is positive definite
over $V_h$.
The seminorm induced by $a_h$ and denoted by
$\LLL\cdot\RRR_h$ is a norm on $V_h$ under this
assumption.
The discretization seeks $u_h\in V_h$ such that
\begin{align}\label{e:bihd}
  a_h(u_h,v_h) = (f,v_h)_{L^2(\Omega)} \quad\text{for all }v_h\in V_h.
\end{align}
It is well known that, for regular partitions,
this is a convergent method on a sequence of
uniformly refined rectangles with maximal mesh size $h$.
The error bound shown in \cite{HuYangZhang2016}
states the quadratic order
$$
 \LLL u-u_h\RRR_h
 \lesssim
 h^2 \|u\|_{H^4(\Omega)}.
$$
In the general case of possibly nonconvex
domains, the assumed regularity is unrealistic,
and local mesh refinement is required for resolving
singularities or the domain geometry.
For rectangular and shape-regular partitions, this necessarily leads to
hanging nodes.
The main question is which continuity properties to
enforce at hanging nodes in the definition of $V_h$
in order to obtain a method with good convergence properties.
Here, we focus on 1-irregular partitions with a maximum of one hanging
node per edge.

\begin{definition}[mesh condition]\label{d:meshcondition}
We say that $\mathcal T$ satisfies the mesh condition if for any irregular
$z\in\mathcal V^{\mathrm{irr}}$ (1) its neighbouring vertices $z_1,z_2$ are regular
and
(2)
any pair $z_1,z_2 \in \mathcal V^{\mathrm{reg}}$ of regular
vertices forming an edge
hosts at most one irregular vertex, i.e.,
$\operatorname{card}(\operatorname{conv}\{z_1,z_2\}
              \cap \mathcal V^{\mathrm{irr}})
              \leq1
$.
\end{definition}
This condition means that every edge containing an irregular vertex
in its interior connects two regular vertices and does not contain
any further irregular vertex.
Figure~\ref{f:noC} shows some configurations excluded by this
condition,
while typical admissible configurations are displayed in
Figure~\ref{f:hanging} or Figure~\ref{f:meshes}.
Let $\mathcal T$ be a partition satisfying the condition
of Definition~\ref{d:meshcondition}.
Such partitions allow for simple $Q_1$ interpolation.
For any function $v$ over $\overline \Omega$ that is continuous in the
regular vertices $\mathcal V^{\mathrm{reg}}$, the interpolation
$Qv$ is the globally continuous
and piecewise bilinear function defined by assigning the nodal value
of $v$ at the regular vertices and, for irregular vertices,
the value of the affine interpolation between
the values at the two neighbouring vertices.
That is, $Q v$ is defined by
\begin{align}\label{e:Qdef}
  Q v (z) =
  \begin{cases}
   v(z)    &\text{if }z\in \mathcal V^{\mathrm{reg}}
    \\
   (\lambda_1+\lambda_2)^{-1}(\lambda_2v(z_1)+\lambda_1v(z_2))
   &\text{if }z\in \mathcal V^{\mathrm{irr}}
               \text{ has neighbours } z_1,z_2 
  \end{cases}
\end{align}
with the weights $\lambda_j:=|z-z_j|$, $j=1,2$.
The approximation properties of $Q$ are discussed in Lemma~\ref{l:Q1stab}
in \S\ref{app:Q1int} of the appendix.

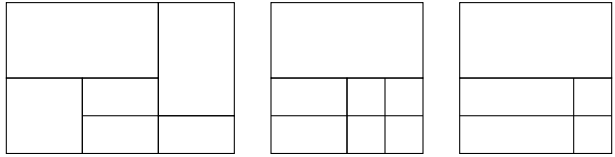
\begin{figure}
 \begin{center}
 \begin{tikzpicture}[scale=1]
  \draw (0,0)--(1,0)--(1,1)--(0,1)--cycle;
  \draw (0,1)--(2,1)--(2,1.5)--(0,1.5)--cycle;
  \draw (1,0)--(2,0)--(2,1)--(1,1)--cycle;
  \draw (1,.5)--(2,.5);
  \draw (2,0)--(3,0)--(3,.5)--(2,.5)--cycle;
  \draw (2,.5)--(3,.5)--(3,1.5)--(2,1.5)--cycle;
 \end{tikzpicture}
 \quad
  \begin{tikzpicture}[scale=1]
  \draw (0,0)--(1,0)--(1,1)--(0,1)--cycle;
  \draw (0,1)--(2,1)--(2,1.5)--(0,1.5)--cycle;
  \draw (1,0)--(2,0)--(2,1)--(1,1)--cycle;
  \draw (0,.5)--(2,.5);
  \draw (1.5,0)--(1.5,1);
 \end{tikzpicture}
 \end{center}
 \caption{Mesh configurations excluded by Definition~\ref{d:meshcondition}.
          Left: some neighbouring vertices are irregular.
          Right: an edge contains more than one irregular vertex.
          }
 \label{f:noC}
\end{figure}

The Adini space $V_h$ over $\mathcal T$ is assumed to be a subspace of
$\widehat V_h$ from \S\ref{s:prelim}.
This fixes the point values in all vertices, the gradient values
in regular vertices, and, for any irregular vertex $z$,
the partial derivative in tangential direction of the edge $E$
containing $z$.
It does not fix the partial derivative at $z$ in the direction
normal to $E$. We will now discuss possible choices
in the next section.

\section{Continuity conditions at hanging nodes}\label{s:hanging}

For a 1-irregular partition $\mathcal T$,
an interior edge with a hanging node $z\in \mathcal V^{\mathrm{irr}}$ will
be shared by three rectangles: one rectangle $T$ for which
$z$ is not a vertex, $z\notin\mathcal V(T)$,
and two rectangles $K_1$, $K_2$ which have
$z$ as a vertex, see Figure~\ref{f:hanging}.
The local degrees of freedom related to $z$ cannot be a global
degree of freedom. Instead, a choice for the value of the function
and its gradient at $z$ has to be made.
For global continuity, it is necessary that
$v$ and the tangential derivative of $v$ are continuous at $z$.
The only freedom that is left is the choice of the derivative
normal to $T$ at $z$.
Any sensible choice must guarantee approximation and consistency.
We ask the assignment of the normal derivative to
be linear, local,
and scaling-invariant:

\begin{definition}[admissible assignment]\label{d:admissible}
Let $z\in\mathcal V^{\mathrm{irr}}$ be an irregular vertex.
There exist exactly three elements $T,K_1,K_2\in\mathcal T$
that contain $z$, where $z$ is a vertex of $K_1$, $K_2$
and belongs to the interior of an edge $E$ of $T$
(see Figure~\ref{f:hanging}) with normal vector $n_E$.
A function $v\in\widehat V_h$ is said to satisfy an admissible
assignment at $z$ if
$$
 \frac{\partial v|_{K_1}}{\partial n_E} (z)
 =
 \frac{\partial v|_{K_2}}{\partial n_E} (z)
 =
 L(v|_T)
$$
for a linear operator $L$ that
(1) is invariant under relabelling coordinates and under linear scaling (homothety),
i.e., if $\hat v$ is defined by $x\mapsto v(\lambda x)$
for a positive real $\lambda$,
then $L \hat v= \lambda Lv$,
and (2) conserves the normal derivative at $z$ for
all quadratic polynomials.
A subspace $V_h\subseteq\widehat V_h$ is said to satisfy
an admissible assignment if any $v_h\in V_h$ satisfies an
admissible assignment at every $z\in\mathcal V^{\mathrm{irr}}$
and if the kernel of $a_h$ over $V_h$ equals $\{0\}$.
\end{definition}

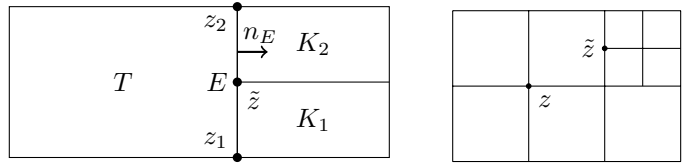
\begin{figure}
 \begin{center}
 \begin{tikzpicture}[scale=2]
  \draw (0,0)--(1.5,0)--(1.5,1)--(0,1)--cycle;
  \node at (.75,.5) []{$T$};
  \draw (1.5,0)--(2.5,0)--(2.5,1)--(1.5,1)--cycle;
  \node at (2,.75) []{$K_2$};
  \node at (2,.25) []{$K_1$};
  \draw (1.5,.5)--(2.5,.5);
    \foreach \y  in {0,.5,1}
      {  \fill (1.5,\y) circle (.2ex);
      }
  \node at (1.5,0) [above left]{$z_1$};
  \node at (1.5,1) [below left]{$z_2$};
  \node at (1.5,.5) [below right]{$\tilde z$};
  \node at (1.5,.5) [left]{$E$};
  \draw[thick,->] (1.5,.7)--(1.7,.7);
  \node at (1.65,.7) [above]{$n_E$};
\end{tikzpicture}
\qquad
 \begin{tikzpicture}[scale=1]
  \coordinate(P) at (1,1);
  \coordinate(H) at (2,1.5);
  \coordinate(N) at (2,1);
  \coordinate(C) at (2.5,1);
  \node at (P) [below right]{$z$};
  \node at (H) [left]{$\tilde z$};
  \node at (N) [below right]{$z'$};
  \node at (C) [above right]{$\check z$};
  \draw[step=1.0,black,thin] (0,0) grid (3,2);
  \draw (2,1.5)--(3,1.5)   (2.5,1)--(2.5,2);
  \draw[fill=black] (P) circle (.2ex);
  \draw[fill=black] (H) circle (.2ex);
  \draw[fill=black] (N) circle (.2ex);
  \draw[fill=black] (C) circle (.2ex);
 \end{tikzpicture}
 \end{center}
 \caption{Left: Configuration with a hanging node $\tilde z$.
  Right: Mesh configuration with a regular vertex $z$ and
         exactly one irregular vertex $\tilde z$
         on $\partial\omega_z$.}
 \label{f:hanging}
\end{figure}

Throughout this work,
we assume that $V_h$ is a linear subspace of $\widehat V_h$ satisfying
an admissible assignment;
in particular $a_h$ is a scalar product on $V_h$.
Then problem \eqref{e:bihd} has a unique solution $u_h\in V_h$,
and the classical a~priori error bound
\cite[Lemma~10.1.7]{BrennerScott2008} known as Berger--Scott--Strang lemma
states that
\begin{equation}\label{e:strang}
 \max\{A,B\} \leq \LLL u-u_h \RRR_h \leq A + B
\end{equation}
for the approximation and consistency errors
$$
  A:=\inf_{v_h\in V_h} \LLL u-v_h\RRR_h
  \quad\text{and}\quad
  B:=\sup_{v_h\in V_h\setminus\{0\}} a_h(u-u_h,v_h) \Big/ \LLL v_h\RRR_h  .
$$
For the method to converge at rate $h^s$ it is necessary that both $A$ and $B$
decrease at least at that rate.
A~priori error estimates are usually formulated on sequences
of uniformly refined meshes. Here, uniform refinement means that
every rectangle is split into four equal sub-rectangles by connecting
the midpoints of opposite edges with straight lines.
If this refinement process is started from an initial
1-irregular partition, eventually the partition will contain
regular vertices $z$ with exactly one irregular vertex on the boundary
of their vertex patch $\omega_z$, as displayed in Figure~\ref{f:hanging}.

We say a method is $O(h^s)$ if there exists a constant $C>0$
such that
$\LLL u-u_h\RRR_h \leq C h^s (\|u\|_{H^4(\Omega)}+\|u\|_{W^{3,\infty}(\Omega)})$
provided the norm of $u$ on the right-hand side is finite.
According to the assignment rule of Definition~\ref{d:admissible},
the space $V_h$ is spanned by global basis functions related to the
degrees of freedom at regular vertices.
The following lemma states that necessary for convergence better than
$O(h)$ is that certain basis functions related to
regular vertices
are continued by $0$ by the admissible assignment.
\begin{lemma}\label{l:lowerbound}
 Let $\mathcal T$ be a 1-irregular partition such that there
 exists a regular vertex $z\in \mathcal V^{\mathrm{reg}}$ with
 $\partial\omega_z\subseteq\Omega$ and
 exactly one irregular vertex $\tilde z\in\mathcal V^{\mathrm{irr}}$
 on the boundary of its vertex patch
 and without irregular vertices inside $\omega_z$ (see Figure~\ref{f:hanging}).
 Let $E\subseteq\partial\omega_z$ denote the edge containing $\tilde z$
 and let $z'$ denote the vertex forming an edge with $z$ and being
 neighbour to $\tilde z$.
 Let $\varphi=\varphi_{z,\alpha}$ 
 with $|\alpha|\leq 1$ denote the Adini basis functions
 with respect to function or derivative evaluation at $z$
 (defined in \S\ref{app:adinibasis} of the appendix) with respect to the multiindex $\alpha$.
 If $\partial_{n_E}\varphi_{z,\alpha}(\tilde z)$ follows an admissible assignment
 and $\varphi_{z,\alpha}$ is not continued by 0 outside $\omega_z$,
 then there exists an $f$ such that the solution
 $u$ belongs to 
 $H^4(\Omega)\cap W^{3,\infty}(\Omega)$,
 but $\LLL u-u_h\RRR_h\geq c_1 h_E - c_2 h_E^2$ 
 with positive numbers $c_1$, $c_2$ independent of the mesh size.
 The same holds true if the admissible assignment depends nontrivially
 on the function value at $z'$.
\end{lemma}
\begin{proof}
 Let $u$ be the solution to \eqref{e:bih} and assume $u\in H^4(\Omega)$.
 Consider the consistency term $B$ from the
 a priori result \eqref{e:strang}.
 Due to \eqref{e:bihd} it satisfies
 $$
    B\geq
    \LLL \varphi\RRR_h^{-1}
     \left (a_h(u,\varphi)- \int_\Omega  f\varphi \right).
 $$
 We follow the notation of Figure~\ref{f:hanging} and denote by $K_1$, $K_2$
 the rectangles with $\tilde z\in \mathcal V(K_1)\cap \mathcal V(K_2)$.
 Clearly, due to the locality in Definition~\ref{d:admissible},
 $\varphi$ vanishes identically outside
 $\omega_z\cup K_1 \cup K_2$.
 We assume that $\varphi$ is not continued by 0
 outside $\omega_z$ and therefore we have with some 
 nonzero real number $c$ that
 $$
   \varphi|_{K_1\cup K_2} = c \psi
 $$
 where $\psi$ is the (local) basis function with
 $\psi=\varphi_{\tilde z,\beta}$ on $K_1\cup K_2$
 with $\beta\neq0$ parallel to $n_E$
 (see \S\ref{app:adinibasis} of the appendix for the notation around the Adini
 basis functions).
  From standard scaling we thus have
 $$
   \|D^2_h\varphi\|_{L^2(K_1\cup K_2)}
   \approx
   |c|
   \quad\text{and}\quad
   \|D^2_h\varphi\|_{L^2(\omega_z)}
   \approx h_E^{|\alpha|-1}.
 $$
 The scaling invariance of Definition~\ref{d:admissible}
 implies $|c|\approx h_E^{|\alpha|-1}$ so that
 $$
   \LLL \varphi \RRR_h \approx h_E^{|\alpha|-1}.
 $$
 Thus,
 $$
  B
   \gtrsim h_E^{1-|\alpha|}
             \left( \int_{\omega_z\cup K_1 \cup K_2}  D^2u : D_h^2 \varphi
             - \int_{\Omega} f\varphi\right)
  .
 $$
 From scaling of $\varphi$ we also have
 $$
  \left| \int_{\Omega}    f\varphi \right|
    \lesssim
    h_E^{1+|\alpha|} \|f\|_{L^2(\Omega)}.
 $$
 Further, it can be computed (see Lemma~\ref{l:d2orthp1}
 in \S\ref{app:adinibasis} of the appendix) that
 $$
 \int_{\omega_z} p\,\partial^2_{jk,h}\varphi = 0
 \quad\text{for any affine }p\in P_1 \text{ and any }j,k=1,2.
 $$
 Standard estimates thus show that
 $|\int_{\omega_z} D^2u:D_h^2\varphi|$ is bounded by a constant times
 $h_E^2\|u\|_{H^4(\Omega)} \LLL \varphi \RRR_h$.
 We thus obtain constants $C_1$, $C_2$ such that
  $$
   B \geq - C_1 
         h_E^2
    (\|u\|_{H^4(\Omega)} + \|f\|_{L^2(\Omega)} )
        +
         C_2  h_E^{1-|\alpha|}\int_{K_1\cup K_2}  D^2u : D^2_h \varphi .
  $$
  Now, by the above requirements, $\varphi|_{K_1\cup K_2}$
  must coincide with $c\psi$.
  We explicitly compute with Lemma~\ref{l:nonzeroint} that
  $$
    \int_{K_1\cup K_2} D^2_h\varphi
    =c\, h_E \begin{bmatrix} \gamma_1&0\\0&\gamma_2\end{bmatrix}
    \quad\text{with }\gamma_1\gamma_2=0 \text{ and }\gamma_1+\gamma_2\neq 0.
  $$
  Without loss of generality, assume that $c\gamma_1>0$.
  Then, if $D^2 u$ is uniformly positive definite
  in a neighbourhood of $K_1\cup K_2$,
  we get the asserted lower bound for $B$
  from the scaling of $|c|$.
  Such $u$
  can be easily obtained by multiplying the function
  $(x^2+y^2)/2$ with a smooth cutoff function
  which is constant $1$ in a neighbourhood of $K_1\cup K_2$,
  so that the Hessian of $u$ equals the unit matrix in that
  region.
  
  In the case that the assignment for the normal derivative at $\tilde z$
  depends nontrivially on the function value at $z'$,
  consider the global basis function $\varphi=\varphi_{z',(0,0)}$.
  Since, by Lemma~\ref{l:dxxdyyorth_globalbf}, the diagonal derivatives
  $\partial^2_{jj,h}\varphi$ 
  are $L^2$ orthogonal to affine functions,
  and since by the tangential continuity of $\varphi$
  the mixed derivative $\partial^2_{xy,h}\varphi$ can be integrated
  by parts against $\partial^2_{xy}u$ without interface jumps,
  we have as above that 
  $|\int_{\Omega} D^2u:D_h^2\varphi| \lesssim h_E^2\|u\|_{H^4(\Omega)} \LLL \varphi \RRR_h$.
  The lower bound then follows with an argument analogous to the one above.
\end{proof}

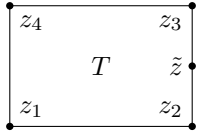
\begin{figure}
 \begin{center}
 \begin{tikzpicture}[scale=1.6]
  \draw (0,0)--(1,0)--(1,1)--(0,1)--cycle;
  \node at (.5,.5) []{$T$};
   \foreach \x/\y  in {0/0,1/0,1/1,0/1,1/.7}
      {  \fill (\x,\y) circle (.2ex);
      }
  \node at (0,0) [above right]{$z_1$};
  \node at (1,0) [above left]{$z_2$};
  \node at (1,1) [below left]{$z_3$};
  \node at (0,1) [below right]{$z_4$};
  \node at (1,.7) [below left]{$\tilde z$};
\end{tikzpicture}
 \end{center}
 \caption{Notation for the reference rectangle used in Lemma~\ref{l:admL}.}
 \label{f:Tref}
\end{figure}

The foregoing Lemma~\ref{l:lowerbound} has the following implication.
 For $|\alpha|\leq1$,
the lower bound in the proof is better than linear only if $c=0$.
If $\partial_{n_E}\varphi_{z,\alpha}(\tilde z)$ follows an admissible assignment
 and the method convergences like $O(h^s)$ with $s>1$
 on quasi-uniform meshes,
 then necessarily $\varphi_{z,\alpha}$
 is continued by 0 outside $\omega_z$.
 Likewise, the normal-derivative degree of freedom must not be coupled
 with the point evaluation at a neighbouring vertex.
 For the assignment operator $L$,
  using the notation for a reference element as displayed in Figure~\ref{f:Tref}
with hanging node $z$,
the lemma states that $L$ must map the basis functions
$\varphi_{z_1,\alpha}$ and $\varphi_{z_4,\alpha}$
 with $|\alpha|\leq 1$ to zero.
The next result shows
that the averaging is the only potentially superlinear
admissible refinement rule that preserves quadratic polynomials.
Recall the Adini basis functions from \S\ref{app:adinibasis}
of the appendix.

\begin{lemma}\label{l:admL}
Consider the reference square $(-1,1)^2$
from Figure~\ref{f:Tref}
 with vertices $z_1,\dots,z_4$ in counterclockwise enumeration
 starting with $z_1=(-1,-1)$.
 Further denote 
   $\tilde z=(r,0)$ with $-1<r<1$.
The only linear admissible (in the sense of Definition~\ref{d:admissible})
map $L:\mathcal A\to\mathbb R$
$L\varphi_{z_2,(0,0)}=0$ as well as
$$
         L\varphi_{z_j,\beta}=0 \text{ for }j=1,4 \text{ and } |\beta|\leq 1,        
         \quad\text{and}\quad
          \partial_x p (\tilde z) = Lp
          \quad\text{for all }p\in P_2
$$
is the averaging
$$
        Lp := \frac12 ( (1-r)\partial_x p(z_2)+ (1+r)\partial_x p(z_3) ).
$$
The space $P_3$ of cubic polynomials is not invariant under such assignment.
\end{lemma}
\begin{proof}
 Since $P_3\subseteq\mathcal A$, any cubic polynomial can be represented
 in the Adini basis as
 $$
    p=\sum_{j=1}^4\sum_{|\alpha|\leq 1}
    \partial^\alpha p(z_j) \varphi_{z_j,\alpha} .
 $$
 From linearity of $L$ and the assumptions, we obtain
 $$
  L p =
         \sum_{j=2,3}\sum_{|\alpha|\leq 1}\partial^\alpha p(z_j) L\varphi_{z_j,\alpha} .
 $$
 Plugging in the polynomials $q=(x-1)y$ and $\tilde q=(x-1)$
 and equating with the $x$-derivative at $\tilde z$ then yields
 $$
  r= Lq = -  L\varphi_{z_2,(1,0)} +  L\varphi_{z_3,(1,0)}
  \quad\text{and}\quad 
  1=L q =   L\varphi_{z_2,(1,0)} +  L\varphi_{z_3,(1,0)}
 $$
 and therefore
 $L\varphi_{z_2,(1,0)}=(1-r)/2$ and $L\varphi_{z_3,(1,0)}=(1+r)/2$.
 Similarly, the polynomials $1$ (the constant) and $(y^2-1)/2$
 lead to
 $$
    0= L\varphi_{z_2,(0,0)}+L\varphi_{z_3,(0,0)}
    \quad\text{and}\quad
   0= -L\varphi_{z_2,(0,1)} + L\varphi_{z_3,(0,1)},
 $$
 which implies $L\varphi_{z_2,(0,0)}=-L\varphi_{z_3,(0,0)}$
 and $L\varphi_{z_2,(0,1)}=L\varphi_{z_3,(0,1)}$.
 Using the assumption $L\varphi_{z_2,(0,0)}=0$
 thus shows $L\varphi_{z_2,(0,0)}=0=L\varphi_{z_3,(0,0)}$.
 Plugging in the polynomial $y$ further yields
 $$
   0=L\varphi_{z_2,(0,1)} + L\varphi_{z_3,(0,1)},
 $$
 which results in
  $L\varphi_{z_2,(0,1)}=0=L\varphi_{z_3,(0,1)}$.
 
 Altogether, the only choice for an admissible map $L$ is the
 asserted averaging.
 It remains to check that this choice cannot preserve all cubic polynomials.
 For the choice $p=(1-y^2)(1-x)$,
 we see that $\partial_x p$ vanishes at all vertices.
 Hence, we have $Lp=0$
 but $\partial_x p (\tilde z) = -(1-r^2)\neq0$.
\end{proof}

The foregoing two Lemmas~\ref{l:lowerbound}--\ref{l:admL}
show that the averaging assignment is the only candidate that
potentially achieves superlinear convergence, which
in particular proves the
lower error bound stated in Theorem~\ref{t:pri}.
Hence, the choice proposed here is to assign the average of the
normal derivatives at the neighbouring vertices:
The global Adini finite element space is defined by
\begin{equation}\label{e:Vh_av}
V_h :=\left\{
         v\in \widehat V_h:
                     Q\nabla v \text{ is continuous at }
                      \mathcal V^{\mathrm{irr}}
      \right\}
\end{equation}
with the bilinear interpolation operator $Q$ defined
in \eqref{e:Qdef}.

\section{Analysis of the consistency error, Proof of Theorem~\ref{t:pri}}\label{s:pri}

Throughout this section, the choice of $V_h$ is fixed through the
averaging rule \eqref{e:Vh_av}.
On any $T$ we introduce local coordinates
$$
 \xi(x,y) = h_x^{-1}(x-m) \quad\text{and}\quad \eta(x,y) = h_y^{-1}(y-m)
$$
ranging from $-1$ to $1$,
where $m$ is the midpoint of $T$
and $h_x$, $h_y$ are the half widths in $x$, $y$ direction, respectively.
By $Q$ we denote the globally continuous and piecewise bilinear
interpolation from \eqref{e:Qdef}.
By the assignment of the hanging node value as in \eqref{e:Vh_av},
the expression $Q\partial_x w$ is well defined for any $w\in V_h$.
We note the following fact, which is essentially contained
in \cite{HuYangZhang2016}.

\begin{lemma}\label{l:Qdx}
 Let $T$ be a rectangle with $E$ an edge orthogonal to the
 $x$-axis. Then any $w\in\mathcal A$ satisfies
 $$
   (1-Q)\partial_x w|_E
   =
   -\frac{h_y^3}{3}\partial^4_{xyyy}w (\eta^3-\eta)
    + \frac{h_y^2}{2}\partial^3_{xyy}w (\eta^2-1).
 $$

\end{lemma}
\begin{proof}
 We express the monomials of the Adini space in terms of $\xi,\eta$.
 It is obvious that $\partial_x P_2$
 and $\partial_x \langle \xi^2\eta,\eta^3\rangle$
 belong to $Q_1 $.
 Further $\partial_x \langle \xi^3,\xi^3\eta\rangle $
 consists of functions that are linear in $\eta$
 and thus are interpolated exactly by $Q$ on $E$.
 Therefore, the only two remaining monomials are  $\xi\eta^2$, $\xi\eta^3$,
 and
 $$
   (1-Q)\partial_x w|_E = (1-Q)\partial_x (a\xi\eta^3 + b\xi\eta^2)|_E
 $$
 with real coefficients $a,b$.
 The chain rule reveals
 $\partial_x \xi=h_x^{-1}$, $\partial_y\eta=h_y^{-1}$.
 Taking derivatives
 of $w$ and comparing coefficients shows that
 $$
   \partial^4_{xyyy} w= \frac{6}{h_x h_y^3} a
   \quad\text{and}\quad
   \partial^3_{xyy}w = \frac{1}{h_x h_y^2} (6a\eta+2b)
  ,
 $$
 which leads to
 $$
  a=\frac{h_x h_y^3}{6} \partial^4_{xyyy} w
  \qquad\text{and}\qquad
  b=\frac{h_x h_y^2}{2}\partial^3_{xyy}w-3a\eta
  .
 $$
 A direct computation of derivatives and interpolation
 leads to
 $$
 (1-Q)\partial_x (a\xi\eta^3+b\xi\eta^2)|_E
 =
 h_x^{-1} \left(
          a (\eta^3-\eta)+b(\eta^2-1)
          \right) .
 $$
 Inserting the values for $a$ and $b$
 in this formula reveals the asserted identity.
\end{proof}

The previous lemma will be essential for bounding the
consistency term in the next lemma.
We denote
\begin{equation}\label{e:Tirrdef}
  \mathcal T^{\mathrm{reg}}
  := \{T\in\mathcal T: \text{all vertices of } T\text{ are regular}  \}
  \qquad\text{and}\quad
  \mathcal T^{\mathrm{irr}}:=\mathcal T\setminus \mathcal T^{\mathrm{reg}} .
\end{equation}

\begin{lemma}\label{l:consist}
 Let the partition $\mathcal T$ satisfy the condition of
 Definition~\ref{d:meshcondition}
 and let $V_h$ be chosen according to \eqref{e:Vh_av}.
 Let $g\in C^1(\overline{\Omega})$ be a piecewise polynomial function
 and $w\in V_h$.
 Then,
 $$
  \sum_{T\in\mathcal T} \int_{\partial T}g (1-Q)\nabla w\cdot n_T
  \lesssim
   (
   \|(1-\Pi_1) g \|_{\cup\mathcal T^{\mathrm{reg}}}
   +
   \|(1-\Pi_0) g \|_{\cup\mathcal T^{\mathrm{irr}}}
   )
   \|D^2_h w \|
  .
 $$
 The constant hidden in the notation $\lesssim$ may depend on
 the polynomial degree of $g$.
\end{lemma}
\begin{proof}
We consider the edges orthogonal to the $x$ or $y$ axis separately.
For a rectangle $T$ we denote by $n_x$ the $x$ component of the
outer unit normal (left $-1$, right $1$, top and bottom $0$).
 Fix any $T$ with local coordinates $(\xi,\eta)$.
 For $(1-Q)\partial_x w$ we use the expression from Lemma~\ref{l:Qdx}
 and the fundamental theorem of calculus so that
 \begin{equation}
  \label{e:intFsplit}
   \int_{\partial T} g (1-Q)\partial_x w n_x
   =
   -\frac{h_y^3}{3}\int_T \partial_x g  \partial^4_{xyyy}w(\eta^3-\eta)
   +\frac{h_y^{2}}{2} \int_T \partial_x g  \partial^3_{xyy}w (\eta^2-1),
 \end{equation}
where is has been used that the fourth and third derivatives of $w$
appearing in Lemma~\ref{l:Qdx} do not depend on $x$.
 Since the function $\eta^3-\eta$ has vanishing average over $T$,
 and since the fourth derivative of $w$ is constant,
 we can use orthogonality to constants and the Cauchy inequality
 for the first term on the right-hand side of \eqref{e:intFsplit}
 to see
 \begin{subequations}
   \label{e:consist-term1}
 \begin{equation}
   -\frac{h_y^3}{3}\int_T \partial_x g  \partial^4_{xyyy}w (\eta^3-\eta)
  \leq
   \frac{h_y^3}{3} \|(1-\Pi_0)\partial_x g\|_T \|\partial^4_{xyyy}w\|_T
   \|\eta^3-\eta\|_{L^\infty(T)}.
 \end{equation}
 With the inverse estimate and $0\leq\eta\leq 1$ we obtain
 that this is bounded by some constant times
 \begin{equation}
   h_y\|(1-\Pi_0)\partial_x g\|_T \|\partial^2_{xy}w\|_T.
 \end{equation}
 \end{subequations}
 Next, consider the second term of \eqref{e:intFsplit}.
 With the global bilinear interpolation $Q$ from \eqref{e:Qdef}
 we obtain with the abbreviation $e:=w-Qw$
 (note that $\partial^3_{xyy} w$ and $\partial^3_{xyy}e$ coincide)
 and integration by parts with respect to $y$ that
 $$
  \frac{h_y^{2}}{2}\int_T \partial_x g (\eta^2-1) \partial^3_{xyy} w
  =
  -\frac{h_y^{2}}{2}\int_T \partial^2_{xy} g (\eta^2-1) \partial^2_{xy} e
  -
  h_y\int_T \partial_x g \eta \partial^2_{xy} e
  .
 $$
 Again, with integration by parts with respect to $x$,
 $$
  -h_y\int_T \partial_x g \eta \partial^2_{xy} e
  =
  h_y\int_T \partial^2_{xx} g \eta \partial_{y} e
  - h_y\int_{\partial T} \partial_x g \eta \partial_{y} e n_x.
 $$
 Integrating by parts along any edge $E$ parallel to the $y$ axis
 with end points $z_-,z_+$
 reveals
 $$
 - h_y\int_{E} \partial_x g \eta \partial_{y} e
 =
 h_y\int_{E} \partial^2_{xy} g \eta  e
 +
 \int_{E} \partial_x g  e
 - h_y (  (\partial_x g  e) (z_+) +(\partial_x g  e) (z_-)).
 $$
 Combining the three foregoing displayed identities yields
\begin{align}
  \label{e:consist-term2}
  \begin{aligned}
 &
 \frac{h_y^{2}}{2}\int_T \partial_x g (\eta^2-1) \partial^3_{xyy} w
  \\
 &\,
 =
  -\frac{h_y^{2}}{2}\int_T \partial^2_{xy} g (\eta^2-1) \partial^2_{xy} e
  +
  h_y\int_T \partial^2_{xx} g \eta \partial_{y} e
  +
  h_y\int_{\partial T} \partial^2_{xy} g \eta  e n_x
  +
  R(T)
 \end{aligned}
 \end{align}
 with
 $$
 R(T):=
  \int_{\partial T} \partial_x g  e n_x
  -
  \sum_{z\in\mathcal V(T)}
  h_y n_x (\partial_x g  e) (z).
 $$
 We note that the inverse inequality implies for any second-order
 derivative $\partial^2_{jk}$ of the polynomial $g|_T$ that
 $$
 \|\partial^2_{jk} g\|_T
 \lesssim h_T^{-1}
 \|(1-\Pi_0)\partial_j g\|_T.
 $$
 The combination of
 \eqref{e:intFsplit}--\eqref{e:consist-term2}
 with this estimate,
 the bound $0\leq\eta\leq 1$, trace and inverse inequalities,
  and the approximation and
 stability properties of $Q$ from Lemma~\ref{l:Q1stab},
 lead to
$$
\int_{\partial T} g (1-Q)\partial_x w n_x
 \lesssim
  h_T\|(1-\Pi_0)\partial_x g\|_T 
 \|D_h^2 w\|_{\omega_T}
  + R(T)
 .
$$
Of course we have
$h_T\|(1-\Pi_0)\partial_x g\|_T\lesssim \|(1-\Pi_1)g\|_T$
for the polynomial $g|_T$.
Considering the sum over all $T$,
since $\partial_x g$ and $e$ are globally continuous
and since $e$ vanishes on $\partial\Omega$,
we have that
$$
\sum_{T\in\mathcal T}  \int_{\partial T} \partial_x g  e n_x= 0 .
$$
We further note that $e(z)=0$ for every regular vertex $z$.
Therefore,
$$
\sum_{T\in\mathcal T} R(T)
\lesssim
\sum_{T\in\mathcal T}
\sum_{z\in\mathcal V(T)\cap\mathcal V^{\mathrm{irr}}}
      | h_T (\partial_x g  e) (z)|.
$$
Given $z\in\mathcal \mathcal V(T)\cap\mathcal V^{\mathrm{irr}}$,
trace and inverse estimates show
with the stability properties of $Q$ from Lemma~\ref{l:Q1stab}
that
$$
  | h_y (\partial_x g  e) (z)|
      \lesssim
        \|(1-\Pi_0) g\|_T \|D^2_h w\|_{\omega_T} .
$$
Combining the above estimates results in the 
asserted estimate
with
$n_T$ replaced by $n_x$. An analogous argument shows the same bound for
$n_T$ replaced by $n_y$, so that eventually the full assertion follows.
\end{proof}

\textbf{\emph{Proof of Theorem~\ref{t:pri}}.}
The abstract a~priori error estimate \eqref{e:strang} shows
that the error is bounded by $A+B$.
We start by bounding $A$.
Let $I_h u$ denote the standard Adini interpolation of $u$
described in \S\ref{app:qiAdini} of the appendix.
On any element $T\in\mathcal T^{\mathrm{reg}}$ with regular
vertices, the standard interpolation bound shows
$\|D_h^2 (u-I_h u)\|_T\lesssim h_T^2 \|D^4 u\|_T$.
If $T\in\mathcal T^{\mathrm{irr}}$ contains an irregular
vertex, $I_h$ preserves quadratic functions and therefore
the standard interpolation bound shows
$\|D_h^2 (u-I_h u)\|_T\lesssim h_T \|D^3 u\|_T$.
Altogether,
$$
 A^2
 \leq
 \|D_h^2 (u-I_h u)\|_{\cup \mathcal T^{\mathrm{reg}}}^2
 +
 \|D_h^2 (u-I_h u)\|_{\cup \mathcal T^{\mathrm{irr}}}^2
 \lesssim
 h^4 \|u\|_{H^4(\cup \mathcal T^{\mathrm{reg}})}^2
 +
 h^2 \|u\|_{H^3(\cup \mathcal T^{\mathrm{irr}})}^2.
$$
For bounding the term $B$,
consider any $v_h\in V_h$ with $\LLL v_h\RRR_h=1$.
Then, the solution property of $u_h$,
integration by parts,
and $\Delta^2 u =f $
show
$$
a(u-u_h,v_h)
=
\sum_{T\in\mathcal T}\int_{\partial T}
            \partial^2_{nn} u \nabla v_h\cdot n_T
=
\sum_{T\in\mathcal T}\int_{\partial T}
            \partial^2_{nn} u (1-Q)\nabla v_h\cdot n_T
$$
because $Q\nabla v_h$ is continuous
and vanishes on the boundary.
Let $\bm g:=\mathcal JD^2u\in [C^1(\overline\Omega)]^{2\times2}$ denote
the (component-wise)
BFS averaging of $D^2 u$ defined in \S\ref{app:bfsaveraging}.
Adding and subtracting $\bm g$ in the above identity results in
$$
a_h(u-u_h,v_h)
=
\sum_{T\in\mathcal T}\int_{\partial T}
            (\partial^2_{nn} u-\bm g_{nn}) (1-Q)\nabla v_h\cdot n_T
+\sum_{T\in\mathcal T}\int_{\partial T}
            \bm g_{nn} (1-Q)\nabla v_h\cdot n_T.
$$
Trace inequalities and the approximation
and discrete stability properties
of $Q$ and $\mathcal J$
from Lemma~\ref{l:Q1stab} and Lemma~\ref{l:bfsaveraging}
bound the first sum on the right-hand side
as follows
$$
\sum_{T\in\mathcal T}\int_{\partial T}
            (\partial^2_{nn} u-\bm g_{nn}) (1-Q)\nabla v_h\cdot n_T
\lesssim
h^2 \|u\|_{H^4(\Omega)} \LLL v_h\RRR_h.
$$
The second term in the above split is bounded with the help of
Lemma~\ref{l:consist}.
A piecewise use of Poincar\'e's inequality
and Lemma~\ref{l:bfsaveraging}
then conclude the proof of the 
first stated upper error bound in Theorem~\ref{t:pri}.
Under uniform mesh refinement, the area covered by elements with
irregular vertices scales like
$$
     \operatorname{meas} (\cup \mathcal T^{\mathrm{irr}}) \lesssim h.
$$
The first error bound and 
the assumed $L^\infty$ bound on the third derivatives thus
imply
$$
\LLL u-u_h\RRR_h^2
\lesssim
 h^4 \|u\|_{H^4(\cup \mathcal T^{\mathrm{reg}})}^2
 +
 h^3 \|u\|_{W^{3,\infty}(\cup \mathcal T^{\mathrm{irr}})}^2
$$
and thus the second stated upper error bound.
The stated lower error bound follows from Lemmas~\ref{l:lowerbound}--\ref{l:admL}.

\section{A posteriori error estimate, Proof of Theorem~\ref{t:post}}\label{s:post}

We define the projection operator
$
 \Pi^{\mathcal T}
$
by
\begin{equation}\label{e:PiTproj}
        \Pi^{\mathcal T} v|_T :=
        \begin{cases}
         \Pi_1 v|_T &\text{if }T\in\mathcal T^{\mathrm{reg}} \\
         \Pi_0 v|_T &\text{if }T\in\mathcal T^{\mathrm{irr}}  .
        \end{cases}
\end{equation}
Any edge is equipped with a fixed normal vector $n_E$ and tangential
vector $t_E$.
The jump across $E$ is denoted by $[\cdot]_E$; for boundary
edges, $[\cdot]_E$ denotes the trace.

For any $T\in\mathcal T$ define the local error estimator contribution by
\begin{subequations}\label{e:estimator}
\begin{align}
 \bm\eta^2(T)
 =
 h_T^4 \|f\|_T^2
 +
 \sum_{j=1}^3
 \sum_{E\in\mathcal E(T)}
 \kappa_{j}^E
 h_T^{2j-3} \left\|
           \left[\frac{\partial^j u_h}{\partial n_E^j}\right]_E
            \right\|_E^2
 +
  \|(1-\Pi^{\mathcal T}) D^2 u_h\|_T^2
 ,
\end{align}
where $\mathcal E(T)$ is the set of edges of $T$ and
\begin{align*}
  \kappa_j^E=\begin{cases} 0 &\text{if }j\geq 2 \text{ and }E\subseteq\partial\Omega,\\
                            1 &\text{otherwise}
              \end{cases}
\end{align*}
is introduced for excluding boundary edges from the sums when
second- and third-order normal derivatives of $u_h$
are considered.
Define  the total error estimator
\begin{align}
\bm\eta = \left(\sum_{T\in\mathcal T}\bm\eta^2(T)\right)^{1/2} .
\end{align}
\end{subequations}

\textbf{\emph{Proof of Theorem~\ref{t:post}}.}
As in \cite{CarstensenGallistlHu2013}, the error is
orthogonally split as follows
$$
\LLL u-u_h\RRR_h^2
=
\left[
\sup_{\varphi\in H^2_0(\Omega)\setminus\{0\}}
\frac{a_h(u-u_h,\varphi)}{\LLL\varphi\RRR_h}\right]^2
+
\min_{v\in H^2_0(\Omega)} \LLL u_h-v\RRR_h^2
.
$$
Since the second term on the right-hand side is directly bounded
by $\bm\eta^2$ after plugging in the BFS averaging $v=\mathcal J_0u_h$
with zero boundary conditions
from \S\ref{app:bfsaveraging} and using
the bound from Lemma~\ref{l:bfsaveraging} and inverse estimates,
it remains to bound the first term.
Let $\varphi\in H^2_0(\Omega)$ with $\LLL\varphi\RRR_h=1$
and denote by $I_h\mathcal J\varphi$ its Adini quasi-interpolation
from \S\ref{app:qiAdini} of the appendix
and abbreviate $\hat\varphi:=\varphi-I_h\mathcal J\varphi$.
Equation~\ref{e:bih} and the discrete solution property \eqref{e:bihd}
yield
$$
a_h(u-u_h,\varphi) = \int_\Omega f\hat\varphi - a_h(u_h,\hat \varphi) .
$$
The first term on the right-hand side is readily bounded by
$\bm\eta$ through \eqref{e:adiniQI}.
For the analysis of the second term, consider its contribution
on any element $T$.
Two integrations by parts reveal
$$
\int_T D^2 u_h : D^2\hat \varphi
=
\int_{\partial T} \partial_{nn}^2 u_h \partial_n \hat\varphi
+
\int_{\partial T} \partial_{nt}^2 u_h \partial_t \hat\varphi
-
\int_{\partial T} (\operatorname{div} D^2u_h)\cdot n \hat\varphi
.
$$
Summing over all elements and noting that $\hat\varphi$ and so
$\partial_t\hat\varphi$ is continuous, we obtain
\begin{equation*}
a_h(u_h,\hat \varphi)
=
\sum_{T\in\mathcal T}
 \int_{\partial T} \partial_{nn}^2 u_h \partial_n \hat\varphi
 +
\sum_{E\in\mathcal E}
\left(
 \int_E
  [\partial_{nt}^2 u_h]_E \partial_t \hat\varphi
-
\int_{E} [\operatorname{div} D^2u_h]_E\cdot n_E \hat\varphi 
\right)
.
\end{equation*}
Standard estimates \cite{Verfuerth2013}
with \eqref{e:adiniQI} bound the last two terms by $\bm\eta$.
In particular, by $\hat\varphi=\partial_t\hat\varphi=0$ on $\partial\Omega$
the boundary edges do not contribute to the sum.
For the analysis of the first sum on the right-hand side,
denote by $\bm g:=\mathcal J D^2_h u_h$ the component-wise
BFS averaging of the piecewise Hessian
from \S\ref{app:bfsaveraging}.
We have with 
$\varphi\in H^2_0(\Omega)$
and
the continuity of
$\bm g$ and $Q\nabla I_h\mathcal J\varphi$ that
$$
\sum_{T\in\mathcal T}
 \int_{\partial T} \partial_{nn}^2 u_h \partial_n \hat\varphi
=
\sum_{T\in\mathcal T}
\left(
\int_{\partial T} (\partial_{nn}^2 u_h - \bm g_{nn}) \partial_n \hat \varphi
 +
 \int_{\partial T} \bm g_{nn} (Q-1) \nabla I_h\mathcal J\varphi \cdot n_T
 \right).
$$
The trace and inverse inequalities and Lemma~\ref{l:consist}
show that this is bounded by a constant times
$$
\sum_{j=1,2}
(\|\partial_{jj}^2 u_h - \bm g_{jj}\|
+
\|(1-\Pi^{\mathcal T})\bm g_{jj}\|
)
,
$$
where we used \eqref{e:adiniQI} and $\|D^2_h \varphi\| = 1$.
Since, obviously,
$$
 \|(1-\Pi^{\mathcal T}) \bm g_{jj}\|
 \leq
 \|(1-\Pi^{\mathcal T}) \partial_{jj}^2 u_h\| 
   + \|\bm g_{jj}-\partial_{jj}^2 u_h\|
$$
we eventually have
$$
\sum_{T\in\mathcal T}
 \int_{\partial T} \partial_{nn}^2 u_h \partial_n \hat\varphi
 \lesssim
 (\|D^2_h u_h - \bm g\|
 +\|(1-\Pi^{\mathcal T}) D^2 u_h\| ) .
$$
The last term is part of (and thus bounded by) $\bm\eta$.
The bound of the first term follows from Lemma~\ref{l:bfsaveraging}.
This concludes the proof of reliability.
The efficiency follows from known arguments
\cite{Verfuerth2013,CarstensenGallistlHu2013}.

\section{Extension to other boundary conditions}\label{s:bc}
This section explains how the proofs shown for clamped boundary 
conditions in the foregoing sections extend to other boundary conditions,
which are relevant in plate bending, flow problems,
or phase separation models.
As the prototypical situation in linear plate theory,
we consider a disjoint partition of $\partial\Omega$
as
$\partial\Omega = \Gamma_C \cup \Gamma_S \cup \Gamma_F$
in clamped ($\Gamma_C$), simply supported ($\Gamma_S$),
and free ($\Gamma_F$) part
and consider the energy space
$$
 V= \{
       v\in H^2(\Omega):
      v=0 \text{ on }\Gamma_C\cup\Gamma_S
     \text{ and }\partial_n v=0 \text{ on } \Gamma_C
   \}
$$
instead of $H^2_0(\Omega)$ from prior sections
where $\partial\Omega=\Gamma_C$ was assumed.
The variational problem is then to find $u\in V$
solving
\eqref{e:bih} with $H^2_0(\Omega)$ replaced by $V$.
The strong form reads
$$
\Delta^2 u = f \text{ in }\Omega
\quad\text{and}\quad
u=0\text{ on }\Gamma_C\cup\Gamma_S
\text{ and }
\partial_nu=0\text{ on }\Gamma_C
$$
with the natural boundary conditions
$$
 \partial^2_{nn} u =0 \text{ on }\Gamma_S\cup\Gamma_F
 \text{ and }
 \partial_n(\partial^2_{nn} u+2\partial^2_{tt} u) = 0 \text{ on }\Gamma_F.
$$
Here $n$ denotes the outer unit vector on $\partial\Omega$
and $t=(-n_2,n_1)$ is a unit tangent vector.
The problem is well posed provided the boundary configuration is
such that $V$ does not contain affine functions.
For the discretization with the Adini FEM we assume that the 
boundary edges match with the decomposition of the boundary
and that $\Gamma_C$ and $\Gamma_C\cup\Gamma_S$ are relatively
closed sets
and replace the definition of $\widehat V_h$
from \S\ref{s:prelim} by
$$
\widehat V_h
:=
C(\overline\Omega)\cap\left\{v\in\mathcal A(\mathcal T) \left|
                \begin{aligned}
                                  &\nabla v \text{ is continuous in the interior
                                  regular vertices of } \mathcal T
                                  \\
                               &v \text{ and } \nabla v\cdot t
                                  \text{ vanish on the vertices of }
                                \Gamma_C\cup\Gamma_S
                                 \\
                            & \nabla v\cdot n
                                  \text{ vanishes on the vertices of }
                                \Gamma_C 
                \end{aligned}
        \right.
       \right\}.
$$
The Adini finite element space is then $V_h$ from
\eqref{e:Vh_av} where the modified version of 
$\widehat V_h$ is used.
The discrete problem seeks $u_h\in V_h$ solving
\eqref{e:bihd} where again this modified $V_h$ is used.

The results of \S\ref{s:hanging} as well as
Lemma~\ref{l:Qdx} hold without modifications.
The error bound of Lemma~\ref{l:consist} holds with
$\cup\mathcal T^{\mathrm{irr}}$
replaced by
$$
(\cup\mathcal T^{\mathrm{irr}})
\cup
\{ T\in\mathcal T: T\cap\Gamma_F\neq\emptyset\}
$$
on the right-hand side.
In the proof, only the analysis of the term 
$\sum_{T\in\mathcal T}R(T)$ requires modifications.
Indeed, since $e$ therein vanishes on $\Gamma_C\cup\Gamma_S$,
but not necessarily on $\Gamma_F$, we have
$$
\sum_{T\in\mathcal T}  \int_{\partial T} \partial_x g  e n_x
=
\sum_{E\subseteq\overline{\Gamma}_F}
     \int_E \partial_x g  e n_x,
$$
where the sum on the right-hand side is over all edges contained
in the closure of $\Gamma_F$.
This term is then bounded with trace and inverse inequalities
and leads to the increased integration domain of $(1-\Pi_0)g$
mentioned above.
Theorem~\ref{t:pri} applies to this situation as well,
again with the same replacement of $\cup\mathcal T^{\mathrm{irr}}$
stemming from the use of Lemma~\ref{l:consist} in the 
proof.
Note that in the proof of Theorem~\ref{t:pri}
the product
$\partial^2_{nn} u Q\nabla v_h\cdot n$
still vanishes on the whole boundary due to the natural
boundary condition satisfied by $u$.
In particular, the asymptotic convergence rate $h^{3/2}$
remains true because the area covered by the element touching
the free boundary $\Gamma_F$ scales like $h$.

For the extension of Theorem~\ref{t:post},
we need to modify the projection $\Pi^{\mathcal T}$ 
from \eqref{e:PiTproj} in so far as 
$\Pi^{\mathcal T}v|_T$ should equal $\Pi_0v|_T$
also on the elements having an edge on $\overline{\Gamma}_F$.
With this definition, the error estimator contribution $\bm\eta^2(T)$
is defined as
\begin{align*}
 \bm\eta^2(T)
 =
 h_T^4 \|f\|_T^2
 &
 +
 \sum_{j=1}^3
 \sum_{E\in\mathcal E(T)}
 \kappa_{j}^E
 h_T^{2j-3} \left\|
           \left[\frac{\partial^j u_h}{\partial n_E^j}\right]_E
            \right\|_E^2
 +
  \|(1-\Pi^{\mathcal T}) D^2 u_h\|_T^2
  \\
  &
   +
   \sum_{\substack{ E\in\mathcal E(T) \\ E\subseteq\overline{\Gamma}_F}}
 h_T^{3} \left\|
           \left[\partial_{n_E}\left(
              \frac{\partial^2 u_h}{\partial n_E^2}+2\frac{\partial^2 u_h}{\partial t_E^2}
              \right)\right]_E
            \right\|_E^2
 ,
\end{align*}
where 
\begin{align*}
  \kappa_j^E=\begin{cases}
                           0 &\text{if }j\in\{2,3\} \text{ and }E\subseteq\overline{\Gamma}_C,\\
                           0 &\text{if }j\in\{1,3\} \text{ and }E\subseteq\overline{\Gamma}_S\cup\overline{\Gamma}_F,\\
                            1 &\text{otherwise} 
              \end{cases}
\end{align*}
such that the residuals of the natural boundary conditions 
on $\Gamma_S$ and $\Gamma_F$
and of the essential boundary conditions on $\Gamma_C$
and $\Gamma_S$ are adequately included.
With these modifications, Theorem~\ref{t:post} holds for the 
situation of general boundary as well.
In the proof, $H^2_0(\Omega)$ needs to be replaced by $V$,
the BFS averaging $\mathcal J$ needs to be defined accordingly 
so that essential boundary conditions correspond to 
$\Gamma_C$ and $\Gamma_S$
(see
 \cite[Proposition~2.5]{Gallistl2015IMA}
 for a similar operator).
The same applies to the Adini interpolation $I_h$.
With these adaptations, the proof of Theorem~\ref{t:post} is very
similar to the one given in \S\ref{s:post}.
As a major modification,
on the edges of $\Gamma_F$, the prior statement 
$\hat\varphi=\partial_t\hat\varphi=0$ does not hold. Instead,
after integration by parts in tangential direction, the 
additional residual enters the error estimator.

\section{Numerical results}\label{s:num}

\subsection{Illustration of Theorem~\ref{t:pri} on quasi-uniform meshes}
\label{ss:result_unif}
We start by numerically illustrating the upper and lower
a~priori error bounds
in an elementary setting with the
square $\Omega=(-1,1)^2$ and $f$ such that the
exact solution is given by the biquartic polynomial
$u=-(x^4-2x^2+1)(y^4-2y^2+1)$.
We consider a regular coarse initial partition consisting of four congruent
squares, and sequences of meshes following two refinement variants.
In the first variant (Variant~I), the coarse mesh is uniformly refined once,
and thereafter only one element containing the point $(0,0)$ is refined,
resulting in an irregular partition.
From this third mesh on, again uniform refinements are performed.
In the second variant (Variant~II), five uniform refinements of the initial mesh are
performed before the seventh mesh is generated by refining only one
element containing $(0,0)$. After that, the refinements are again uniform.

For the sake of a consistent presentation, in convergence history plots
of this work, the horizontal axis displays
the squareroot of the number of degrees of freedom $\mathtt{ndof}$.
For these quasi-uniform meshes in this example, that quantity
is proportional to $h$, the parameter used in Theorem~\ref{t:pri}.
Starting from the third mesh in the sequence, the partitions of Variant~I
are irregular. Theorem~\ref{t:pri} predicts the error to decrease as
$h^{3/2}$ for $V_h$ as in \eqref{e:Vh_av} and not better than $h$
for any other choice.
Up to the seventh mesh, the partitions of Variant~II are uniform, so that
Theorem~\ref{t:pri} states errors of the order $h^2$ on these regular
partitions. After the local refinement, the partitions are irregular,
and Theorem~\ref{t:pri} predicts the error for $V_h$ as in \eqref{e:Vh_av}
to gradually deteriorate to $h^{3/2}$, while any other method must immediately
deteriorate from $O(h^2)$ to $O(h)$.

Figure~\ref{f:conv_priori} (left) displays the convergence history of the
$\LLL\cdot\RRR_h$ error with respect to the squareroot
of the number of degrees of freedom $\mathtt{ndof}$
(the dashed lines showing the asymptotic rates are labelled with
powers of the proportional parameter $h$ used in Theorem~\ref{t:pri}).
Two assignments are compared: the assignment \eqref{e:Vh_av},
abbreviated by $V_h$ in the legend, and the enforcement of
strong continuity of the normal derivative in irregular vertices,
abbreviated by `hard'.
The results are as expected:
for Variant~I, the `hard' interpolation results in
convergence $O(h)$, while the method \eqref{e:Vh_av} reaches the
predicted $O(h^{3/2})$.
For Variant~II,
as soon as a single element is refined and, thus, the partition becomes
irregular, the `hard' interpolation method immediately deteriorates to $O(h)$.

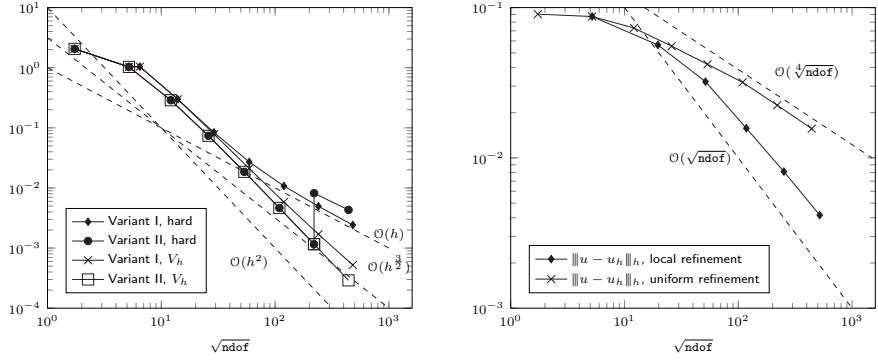
\begin{figure}
\pgfplotstableread{
   sqrtndofA                sqrtndofB                errA                     errB
   1.7320508075688772e+00   1.7320508075688772e+00   2.0591764199176046e+00   2.0591764199176046e+00
   5.1961524227066320e+00   5.1961524227066320e+00   1.0301870180234125e+00   1.0301870180234125e+00
   6.4807406984078604e+00   1.2124355652982141e+01   1.0408233668239601e+00   2.8800754816898716e-01
   1.3964240043768941e+01   2.5980762113533160e+01   2.9809148952765230e-01   7.3756874928884331e-02
   2.9034462281915950e+01   5.3693575034635195e+01   8.4018698630546756e-02   1.8545663212466982e-02
   5.9219929077971713e+01   1.0911920087683927e+02   2.7076295256141362e-02   4.6429767717423827e-03
   1.1961187232043481e+02   2.1997045256124741e+02   1.0764653343088104e-02   1.1611492073577174e-03
   2.4040590674939747e+02   2.2000454540758926e+02   4.9583075182261472e-03   8.1792817454156008e-03
   4.8199896265448541e+02   4.4172729143669625e+02   2.4204452211083146e-03   4.3152292459655269e-03
    }\unifhardI
\pgfplotstableread{
   sqrtndofA                sqrtndofB                errA                     errB
   1.7320508075688772e+00   1.7320508075688772e+00   2.0591764199176046e+00   2.0591764199176046e+00
   5.1961524227066320e+00   5.1961524227066320e+00   1.0301870180234125e+00   1.0301870180234125e+00
   6.4807406984078604e+00   1.2124355652982141e+01   1.0433594478581825e+00   2.8800754816898716e-01
   1.3964240043768941e+01   2.5980762113533160e+01   2.9807751680941419e-01   7.3756874928884331e-02
   2.9034462281915950e+01   5.3693575034635195e+01   7.8992566072808329e-02   1.8545663212466982e-02
   5.9219929077971713e+01   1.0911920087683927e+02   2.1075215530947076e-02   4.6429767717423827e-03
   1.1961187232043481e+02   2.1997045256124741e+02   5.8258067324902449e-03   1.1611492073577174e-03
   2.4040590674939747e+02   2.2000454540758926e+02   1.6965534353542029e-03   1.1611867402397775e-03
   4.8199896265448541e+02   4.4172729143669625e+02   5.2326441128610864e-04   2.9029535204012159e-04
    }\unifQ

	\begin{tikzpicture}[scale=.7]
		\begin{loglogaxis}[legend pos=south west,legend cell align=left,
			legend style={fill=none},
			ymin=1e-4,ymax=1e1,
			xmin=1e0,xmax=1.6e3,
			ytick={1e-4,1e-3,1e-2,1e-1,1e0,1e1},
			xtick={1e0,1e1,1e2,1e3,1e4}]
			\pgfplotsset{
				cycle list={%
					{Black, mark=diamond*, mark size=2pt},
					{Black, mark=*, mark size=2pt},
					{Black, mark=x, mark size=3pt},
					{Black, mark=square, mark size=3pt},
					{Cyan, mark=x, mark size=2.5pt},
					{Cyan, mark=diamond, mark size=3pt},
					{Cyan, mark=*, mark size=1.5pt},
					{Orange, mark=square*, mark size=2.5pt},
					{Orange, mark=square*, mark size=2pt},
					{Orange, mark=square*, mark size=1.5pt},
				},
				legend style={
					at={(0.05,.05)}, anchor=south west},
				font=\sffamily\scriptsize,
				xlabel=$\sqrt{\mathtt{ndof}}$,ylabel=
			}
			\addplot+ table[x=sqrtndofA,y=errA]{\unifhardI};
			\addlegendentry{Variant I, hard}
			\addplot+ table[x=sqrtndofB,y=errB]{\unifhardI};
			\addlegendentry{Variant II, hard }
			\addplot+ table[x=sqrtndofA,y=errA]{\unifQ};
			\addlegendentry{Variant I, $V_h$}
			\addplot+ table[x=sqrtndofB,y=errB]{\unifQ};
			\addlegendentry{Variant II, $V_h$}
			\addplot+ [color=Black,dashed,mark=none] coordinates{(1,1) (1e3,1e-3)};
			\node(z) at  (axis cs:1000,1e-3)
			[above] {$O (h)$};
            \addplot+ [color=Black,dashed,mark=none] coordinates{(1,3.1623) (1e3,1e-4)};
			\node(z) at  (axis cs:1000,5e-4)
			[below] {$O (h^{\frac32})$};
			\addplot+ [color=Black,dashed,mark=none] coordinates{(1,10) (1e3,1e-5)};
			\node(z) at  (axis cs:100,1e-3)
			[below left] {$O (h^{2})$};
		\end{loglogaxis}
	\end{tikzpicture}
\hfil
	\begin{tikzpicture}[scale=.7]
\pgfplotstableread{
   sqrtndofAD               Q1AD
   5.1961524227066320e+00   8.7260770129033993e-02
   1.9672315572906001e+01   5.6411851544589754e-02
   5.1526692111953004e+01   3.2166942498869262e-02
   1.1779218989389746e+02   1.5748196482369217e-02
   2.5160087440229614e+02   8.1086668373767951e-03
   5.1982593240430015e+02   4.1614812327366900e-03
   }\CircLOC
\pgfplotstableread{
   sqrtndofUN               Q1UN
   1.7320508075688772e+00   9.0415467437701347e-02
   5.1961524227066320e+00   8.7260770129033993e-02
   1.2124355652982141e+01   7.3143255499293805e-02
   2.5980762113533160e+01   5.5356053009262565e-02
   5.3693575034635195e+01   4.1952164407156357e-02
   1.0911920087683927e+02   3.1829766836920445e-02
   2.1997045256124741e+02   2.2403884127015543e-02
   4.4167295593006372e+02   1.5659673967495489e-02
   }\CircUN

		\begin{loglogaxis}[legend pos=south west,legend cell align=left,
			legend style={fill=white},
			ymin=1e-3,ymax=1e-1,
			xmin=1e0,xmax=1.6e3,
			ytick={1e-3,1e-2,1e-1,1e0,1e1,1e2},
			xtick={1e0,1e1,1e2,1e3}]
			\pgfplotsset{
				cycle list={%
					{Black, mark=diamond*, mark size=2pt},
					{Black, mark=x, mark size=3pt},
					{Black, mark=triangle, mark size=3pt},
					{Cyan, mark=x, mark size=2.5pt},
					{Cyan, mark=diamond, mark size=3pt},
					{Cyan, mark=*, mark size=1.5pt},
					{Orange, mark=square*, mark size=2.5pt},
					{Orange, mark=square*, mark size=2pt},
					{Orange, mark=square*, mark size=1.5pt},
				},
				legend style={
					at={(0.05,.05)}, anchor=south west},
				font=\sffamily\scriptsize,
				xlabel=$\sqrt{\mathtt{ndof}}$,ylabel=
			}
			\addplot+ table[x=sqrtndofAD,y=Q1AD]{\CircLOC};
			\addlegendentry{$\LLL u-u_h\RRR_h$, local refinement}
			\addplot+ table[x=sqrtndofUN,y=Q1UN]{\CircUN};
			\addlegendentry{$\LLL u-u_h\RRR_h$, uniform refinement}
			\addplot+ [color=Black,dashed,mark=none] coordinates{(1e1,1e-1) (1e3,1e-3)};
			\node(z) at  (axis cs:1e2,1e-2)
			[left] {$O (\sqrt{\mathtt{ndof}})$};
			\addplot+ [color=Black,dashed,mark=none] coordinates{(1.5e1,1e-1) (1.5e3,1e-2)};
			\node(z) at  (axis cs:4e2,3e-2)
			[above] {$O (\sqrt[4]{\mathtt{ndof}})$};
		\end{loglogaxis}
	\end{tikzpicture}

	\caption{Left: Numerical illustration of Theorem~\ref{t:pri}
	         with the errors $\LLL u-u_h\RRR_h$ for a smooth $u$ on the
	         unit square, setting of \S\ref{ss:result_unif}.
	         Right: Convergence history
	          for the disk domain from \S\ref{ss:results_Circ}.
		\label{f:conv_priori}
	}
\end{figure}

\subsection{Approximation of a curvilinear domain}
\label{ss:results_Circ}
As an example for local resolution of a curved boundary,
consider $\Omega$ as the unit disk with $f=1$
and the exact solution given by
$u(x,y)=2^{-6}(x^2+y^2-1)^2$
and discretization by the averaging assignment
\eqref{e:Vh_av}.
For the domain approximation, we consider a partition of the
square $(-1,1)^2$ where all degrees of freedom related to vertices
outside $\Omega$ are set to zero.
On a sequence of uniformly refined meshes this results in
convergence of order $h^{1/2}$ as can be seen in the convergence
history of Figure~\ref{f:conv_priori} (right).
In a locally refined variant, from the $j$-th uniform refinement of
the background mesh, the actual partition is generated by repeating
$j$ times: mark all elements that touch the boundary $\partial\Omega$
for local refinement and generate the smallest $1$-irregular partition
where the marked elements are refined.
Figure~\ref{f:meshes} (left) shows an instance of such a mesh;
and the convergence history showing that this local refinement variant improves
the convergence order to $\sqrt{\mathtt{ndof}}$.

\begin{figure}
 \begin{center}
 \includegraphics[width=.3\textwidth]{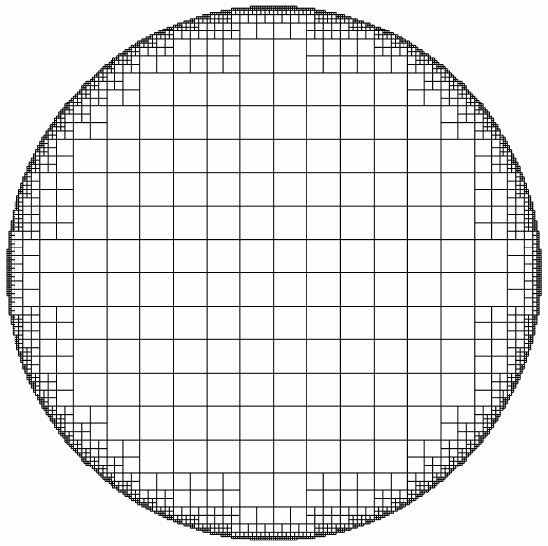}
 \hfill
 \includegraphics[width=.3\textwidth]{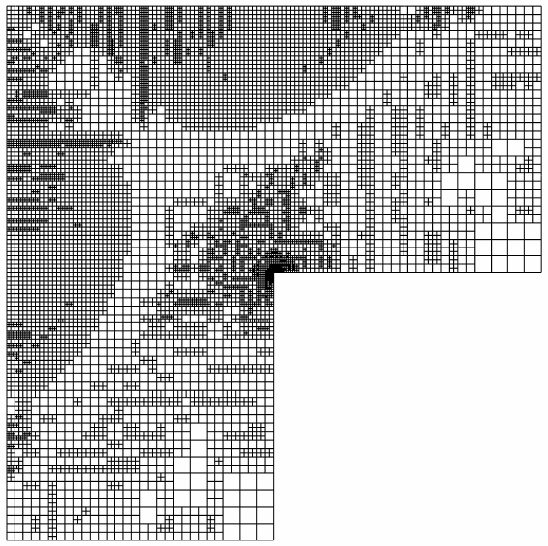}
 \hfill
 \includegraphics[width=.3\textwidth]{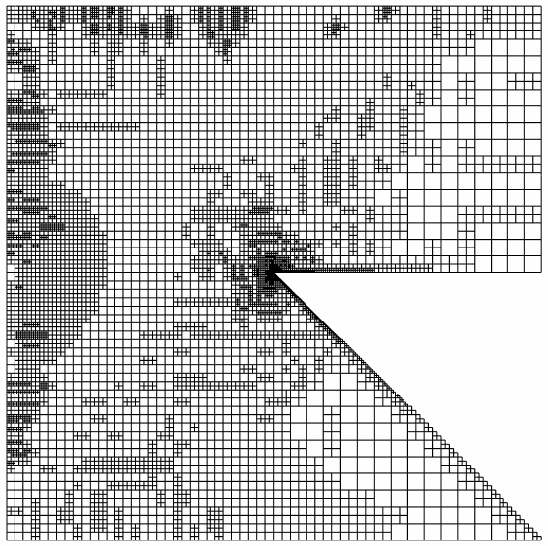}
 \end{center}
 	\caption{
 	  Left: Locally resolved disk, 13\,875 degrees of freedom.
      Middle: Adaptive mesh of the
 	    L-domain, $37\,119$ degrees of freedom, level $17$.
 	  Right: Adaptive mesh of the cusp domain,
 	             $32\,130$ degrees of freedom, level $19$.}
		\label{f:meshes}
\end{figure}

\subsection{Adaptive mesh refinement}\label{ss:results_L}
In this experiment, we consider the error estimator $\bm \eta$
with its local contributions $\bm\eta(T)$ as
a refinement indicator in an adaptive mesh-refinement algorithm
with D\"orfler marking and bulk parameter chosen as 1/2
in a standard adaptivity loop \cite{Verfuerth2013}.
We consider the L-shaped domain
$\Omega = \left(-1,1\right)^2 \setminus \left([0,1]\times[-1,0]\right)$.
With $\alpha=0.5444837\dots$ and $\omega=3\pi/2$,
the exact singular solution from \cite[p.\ 107]{Grisvard1992}
reads in polar coordinates as
\begin{equation}\label{e:GrisvardExactSolution}
u(r,\theta)=(r^2\cos^2\theta-1)^2(r^2\sin^2\theta-1)^2r^{1+\alpha}g(\theta),
\end{equation}
with the function
\begin{align*}
	g(\theta) =\left(\frac{s_-(\omega)}{\alpha-1}-\frac{s_+(\omega)}{\alpha+1} \right)(c_-(\theta)-c_+(\theta))
	-\left(\frac{s_-(\theta)}{\alpha-1}-\frac{s_+(\theta)}{\alpha+1} \right)(c_-(\omega)-c_+(\omega))
\end{align*}
and the abbreviations $s_\pm(z)=\sin((\alpha\pm1)z)$
and $c_\pm(z)=\cos((\alpha\pm1)z)$.
The convergence history with respect to the squareroot of the number of degrees
of freedom ($\mathtt{ndof}$) is displayed in Figure~\ref{f:convLZ} (left).
As expected, uniform mesh refinement converges with the suboptimal rate
dictated by $\alpha$. Adaptive mesh refinement with the averaging assignment
\eqref{e:Vh_av} recovers first-order convergence,
while the variant enforcing the `hard' interpolation constraint performs
poorly (on the same adaptive meshes)
because its error is bounded from below by certain elements
of large mesh size (Theorem~\ref{t:pri}),
see the adaptive mesh displayed in
Figure~\ref{f:meshes} (middle).

\begin{figure}
\begin{minipage}{.53\textwidth}
	\begin{tikzpicture}[scale=.7]
	\pgfplotstableread{
   sqrtndofAD               Q1AD                     hardAD                   etaAD
   3.8729833462074170e+00   4.3022478091840046e+00   4.3022478091840046e+00   5.8326987378375613e+01
   6.9282032302755088e+00   3.9637917540676040e+00   4.0039195388227711e+00   3.7651689480944540e+01
   8.1240384046359608e+00   2.9998620977409653e+00   3.0491925708286676e+00   2.6643715799263500e+01
   9.6436507609929549e+00   1.6238067238679614e+00   1.6093109480413590e+00   2.2796704705498648e+01
   1.3856406460551018e+01   1.6598615774849976e+00   1.6645280697226394e+00   1.6412803329917363e+01
   1.6248076809271922e+01   1.5239103990749137e+00   1.5125770134438898e+00   1.1015863709559374e+01
   1.9748417658131498e+01   9.5238767873599017e-01   9.0702396862648493e-01   7.7470364419788780e+00
   2.7221315177632398e+01   6.8877149441369512e-01   6.8737930712371220e-01   5.8152934717484008e+00
   3.4073450074801642e+01   5.4968271853636219e-01   5.4961123191142425e-01   4.6024024455016557e+00
   4.0804411526206330e+01   4.2692589969079775e-01   3.4680579314284682e-01   3.4979494251763756e+00
   5.0467811523782167e+01   2.8443304592478286e-01   2.3721232171436160e-01   2.5490173723868992e+00
   6.6543219038456499e+01   2.1705322838343449e-01   2.5926346141129691e-01   1.8909608516247518e+00
   8.0665977958492519e+01   1.6987086782544378e-01   2.2711809441439992e-01   1.4413401990439751e+00
   9.8224233262469397e+01   1.2729997723861303e-01   1.6828644345114996e-01   1.1036444762269590e+00
   1.2342609124492276e+02   9.4966992315919332e-02   1.8455235317038923e-01   8.9172465960258773e-01
   1.5553777676178865e+02   8.1423437421030628e-02   1.6129066998595804e-01   7.3917160737521792e-01
   1.9266291807195282e+02   6.3695977384618546e-02   1.8240926376382408e-01   5.9499673592239666e-01
    }\LshapeAD
\pgfplotstableread{
   sqrtndofUN               Q1UN                     hardUN                   etaUN
   3.8729833462074170e+00   4.3022478091840046e+00   4.3022478091840046e+00   5.8263668999475733e+01
   9.9498743710661994e+00   1.5788650946306912e+00   1.5788650946306912e+00   2.2690864451513431e+01
   2.1977260975835911e+01   7.9281859223301021e-01   7.9281859223301021e-01   8.8240815096899432e+00
   4.5989129150267672e+01   5.0551098045540577e-01   5.0551098045540577e-01   4.6859887638893447e+00
   9.3994680700558789e+01   3.4217676500145439e-01   3.4217676500145439e-01   3.0271052792969160e+00
   1.8999736840282816e+02   2.3404169829031551e-01   2.3404169829031551e-01   2.0529925295843419e+00
    }\LshapeUN

		\begin{loglogaxis}[legend pos=north east,legend cell align=left,
			legend style={fill=white},
			ymin=1e-2,ymax=1e2,
			xmin=1e0,xmax=1.6e3,
			ytick={1e-2,1e-1,1e0,1e1,1e2},
			xtick={1e0,1e1,1e2,1e3}]
			\pgfplotsset{
				cycle list={%
					{Black, mark=diamond*, mark size=2pt},
					{Black, mark=*, mark size=2pt},
					{Black, mark=x, mark size=3pt},
					{Black, mark=triangle, mark size=3pt},
					{Black, mark=square, mark size=3pt},
					{Black, mark=pentagon*, mark size=2pt},
					{Cyan, mark=*, mark size=1.5pt},
					{Orange, mark=square*, mark size=2.5pt},
					{Orange, mark=square*, mark size=2pt},
					{Orange, mark=square*, mark size=1.5pt},
				},
				legend style={
					at={(.8,1)}, anchor=north west},
				font=\sffamily\scriptsize,
				xlabel=$\sqrt{\mathtt{ndof}}$,ylabel=
			}
			\addplot+ table[x=sqrtndofAD,y=Q1AD]{\LshapeAD};
			\addlegendentry{error, adapt, $V_h$}
			\addplot+ table[x=sqrtndofAD,y=etaAD]{\LshapeAD};
			\addlegendentry{$\bm\eta$, adapt, $V_h$}
			\addplot+ table[x=sqrtndofUN,y=Q1UN]{\LshapeUN};
			\addlegendentry{error, unif, $V_h$}
			\addplot+ table[x=sqrtndofUN,y=etaUN]{\LshapeUN};
			\addlegendentry{$\bm\eta$, unif, $V_h$}
			\addplot+ table[x=sqrtndofUN,y=hardUN]{\LshapeUN};
			\addlegendentry{error, unif, hard}
			\addplot+ table[x=sqrtndofAD,y=hardAD]{\LshapeAD};
			\addlegendentry{error, adapt, hard}
			\addplot+ [color=Black,dashed,mark=none] coordinates{(1,1e1) (1e3,1e-2)};
			\node(z) at  (axis cs:10,3e-1)
			[above] {$O (\sqrt{\mathtt{ndof}})$};
			\addplot+ [color=Black,dashed,mark=none] coordinates{(1,6) (1e3,0.6*0.2333)};
			\node(z) at  (axis cs:600,6e-1)
			[below] {$O (\sqrt{\mathtt{ndof}}^{0.544})$};
		\end{loglogaxis}
	\end{tikzpicture}
\end{minipage}
%
    \begin{minipage}{.43\textwidth}
	\begin{tikzpicture}[scale=.7]
\pgfplotstableread{
   sqrtndofAD               Q1AD                     hardAD                   etaAD
   1.7320508075688772e+00   4.0107601950656715e+00   0.0000000000000000e+00   0.0000000000000000e+00
   3.4641016151377544e+00   4.0168889329175332e+00   0.0000000000000000e+00   0.0000000000000000e+00
   4.8989794855663558e+00   2.4705310834020699e+00   0.0000000000000000e+00   1.9392317451148344e+01
   6.4807406984078604e+00   1.5718262633644728e+00   0.0000000000000000e+00   2.1550520880366683e+01
   8.4852813742385695e+00   1.2418742870714869e+00   0.0000000000000000e+00   1.5642073430929887e+01
   1.0246950765959598e+01   9.7583841575292019e-01   0.0000000000000000e+00   9.8496049914437780e+00
   1.3856406460551018e+01   6.5276439431244104e-01   0.0000000000000000e+00   8.6187958554803856e+00
   1.7058722109231979e+01   5.0086222055453355e-01   0.0000000000000000e+00   6.2873627951617497e+00
   2.0199009876724155e+01   4.5105930410555062e-01   0.0000000000000000e+00   4.3565691152118484e+00
   2.5865034312755125e+01   3.6098458624741520e-01   0.0000000000000000e+00   3.2913774540078755e+00
   3.1749015732775089e+01   2.4085507259908803e-01   0.0000000000000000e+00   2.2541634662881842e+00
   3.9458839313897720e+01   1.8356192349145059e-01   0.0000000000000000e+00   1.6380303380642838e+00
   4.8805737367649719e+01   1.3414677662021468e-01   0.0000000000000000e+00   1.2088380662407567e+00
   6.2928530890209096e+01   9.7217024152665463e-02   0.0000000000000000e+00   8.9635766996550392e-01
   7.7711003081931707e+01   7.5097580693276350e-02   0.0000000000000000e+00   6.6698805013911622e-01
   9.4535707539532382e+01   6.4318620217875849e-02   0.0000000000000000e+00   5.4127663978426732e-01
   1.1671760792613941e+02   4.8732553960678091e-02   0.0000000000000000e+00   4.2214478866926936e-01
   1.4666287873896380e+02   3.9069885141247720e-02   0.0000000000000000e+00   3.4934175435671089e-01
   1.7924843095547587e+02   3.1827594570709071e-02   0.0000000000000000e+00   2.9026276988512451e-01

   }\ZshapeAD
\pgfplotstableread{
   sqrtndofUN               Q1UN                     hardUN                   etaUN
   1.7320508075688772e+00   4.0107601950656715e+00   0.0000000000000000e+00   0.0000000000000000e+00
   5.1961524227066320e+00   1.5694176112035383e+00   0.0000000000000000e+00   2.2117623324641922e+01
   1.2124355652982141e+01   6.2335009151984688e-01   0.0000000000000000e+00   8.8361487184243135e+00
   2.5980762113533160e+01   3.4723312715891186e-01   0.0000000000000000e+00   3.8641829172953677e+00
   5.3693575034635195e+01   2.3465298050892577e-01   0.0000000000000000e+00   2.3020952769501952e+00
   1.0911920087683927e+02   1.6525822020680153e-01   0.0000000000000000e+00   1.5715856738028864e+00
   2.1997045256124741e+02   1.1720181869989550e-01   0.0000000000000000e+00   1.1039190212037320e+00
   }\ZshapeUN
		\begin{loglogaxis}[legend pos=south west,legend cell align=left,
			legend style={fill=none},
			ymin=1e-2,ymax=1e2,
			xmin=1e0,xmax=1.6e3,
			ytick={1e-2,1e-1,1e0,1e1,1e2},
			xtick={1e0,1e1,1e2,1e3}]
			\pgfplotsset{
				cycle list={%
					{Black, mark=diamond*, mark size=2pt},
					{Black, mark=*, mark size=2pt},
					{Black, mark=x, mark size=3pt},
					{Black, mark=triangle, mark size=3pt},
					{Cyan, mark=x, mark size=2.5pt},
					{Cyan, mark=diamond, mark size=3pt},
					{Cyan, mark=*, mark size=1.5pt},
					{Orange, mark=square*, mark size=2.5pt},
					{Orange, mark=square*, mark size=2pt},
					{Orange, mark=square*, mark size=1.5pt},
				},
				legend style={
					at={(-0.11,.5)}, anchor=east},
				font=\sffamily\scriptsize,
				xlabel=$\sqrt{\mathtt{ndof}}$,ylabel=
			}
			\addplot+ table[x=sqrtndofAD,y=Q1AD]{\ZshapeAD};
			\addplot+ table[x=sqrtndofAD,y=etaAD]{\ZshapeAD};
			\addplot+ table[x=sqrtndofUN,y=Q1UN]{\ZshapeUN};
			\addplot+ table[x=sqrtndofUN,y=etaUN]{\ZshapeUN};
%
			\addplot+ [color=Black,dashed,mark=none] coordinates{(1,1e1) (1e3,1e-2)};
			\node(z) at  (axis cs:5e2,4e-2)
			[above] {$O (\sqrt{\mathtt{ndof}})$};
			\addplot+ [color=Black,dashed,mark=none] coordinates{(1,20) (1e3,0.6109)};
			\node(z) at  (axis cs:600,3e0)
			[below] {$O (\sqrt{\mathtt{ndof}}^{0.505})$};
		\end{loglogaxis}
	\end{tikzpicture}
	\end{minipage}
	\caption{%
	Convergence history of the error $\LLL u-u_h\RRR_h$
	and the estimator $\bm\eta$ for the L-shaped (left) and the cusp (right) domain
	from \S\S\ref{ss:results_L}--\ref{ss:results_Z}.
	\label{f:convLZ}
	}
\end{figure}
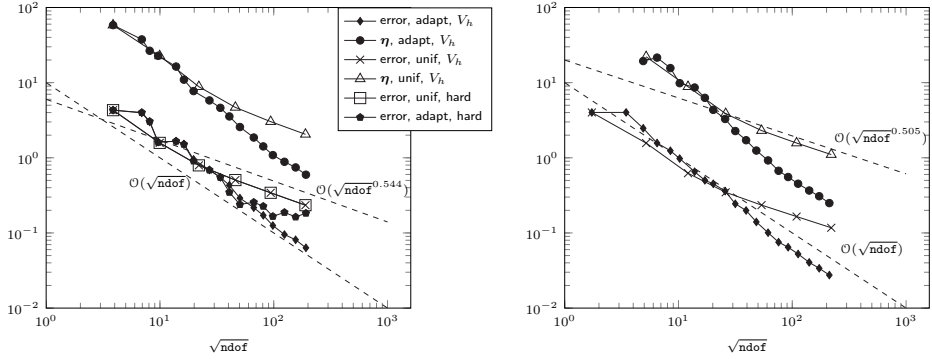

\subsection{Approximation of a non-rectilinear domain}
\label{ss:results_Z}
As an example for adaptive resolution of a non-rectilinear domain,
consider the corner domain
$\Omega=(-1,1)^2\setminus\mathrm{conv}\{(0,0),(1,-1),(1,0)\}$
with exact solution given by \eqref{e:GrisvardExactSolution}
for the
parameters $\alpha=0.50500969\dots$ and $\omega=7\pi/4$.
The line with angle $7\pi/4$ describes the non-rectilinear part of
the boundary.
We use an interior approximation with rectangles
and add on elements $T$ touching the boundary the local error estimator
contribution $h_T^2\bm\eta^2$ to $\eta^2(T)$ in the marking process.
This accounts for the error by the boundary approximation.
Figure~\ref{f:convLZ} (right)
shows the convergence history.
As in the previous example, adaptive mesh refinement improves the
reduced convergence observed on uniform meshes.
An adaptive mesh is displayed in Figure~\ref{f:meshes} (right).

\subsection{Mixed boundary conditions}
\label{ss:results_Lmixed}

In this example, the generalization of
the method and the a~posteriori error estimator from
\S\ref{s:bc} to more general boundary conditions are applied.
On the L-shaped domain, the following mixed configuration
in simply supported and clamped parts is chosen
$$
 \Gamma_S=\big(\{0\}\times(-1,0]\big)
           \cup \big( [0,1)\times\{0\} \big)
 \quad\text{and}\quad
 \Gamma_C=\Omega\setminus\Gamma_S.
$$
and the right-hand $f$ side is chosen according
to the exact solution
$$
u(x,y)
=
r^{4/3} \sin (4\theta/3)
(1 - x^2)^2 (1 - y^2)^2
$$
with polar coordinates $r$, $\theta$ as functions of $x,y$.
This function belongs to
$H^{7/3-\nu}(\Omega)$ for every $\nu>0$ but not
for $\nu=0$
and it is known \cite{BlumRannacher1980}
that this is the generic singularity for
the simply supported boundary condition at an opening
angle $3\pi/2$.
The convergence history of uniform and adaptive mesh refinement
and an adaptive mesh are displayed Figure~\ref{f:convLmix}.
As expected, uniform mesh refinement cannot converge with a
rate higher than 1/3, whereas the adaptive version leads to
errors decreasing like $\sqrt{\mathtt{ndof}}$.

\begin{figure}
\begin{minipage}{.53\textwidth}
	\begin{tikzpicture}[scale=.7]
	\pgfplotstableread{
   sqrtndofAD               Q1AD                     etaAD
   4.1231056256176606e+00   1.4974963855282994e+00   1.6850238315415169e+01
   7.6811457478686078e+00   1.2695579845851015e+00   9.6584332950048690e+00
   9.5393920141694561e+00   1.1428130332226047e+00   7.6326252746854140e+00
   1.2124355652982141e+01   7.6440753637382164e-01   5.7833609105196393e+00
   1.6124515496597098e+01   6.2149953406865976e-01   4.4755202615574641e+00
   1.9235384061671343e+01   4.9671308631429478e-01   3.4480736218153125e+00
   2.2045407685048602e+01   4.0179371750177667e-01   2.6863249709452255e+00
   2.8965496715920477e+01   2.7821741274534306e-01   1.9739259752429708e+00
   3.4423828956117013e+01   2.2687500804461053e-01   1.6023018671692595e+00
   3.9281038682804713e+01   1.9325433804716680e-01   1.3779514390040781e+00
   4.3943145085439660e+01   1.5665624558016417e-01   1.1215204605237823e+00
   4.8187135212627034e+01   1.4069278866783549e-01   1.0094481249172398e+00
   5.4845236803208351e+01   1.1722669293980809e-01   8.5149594540344054e-01
   6.8330081223426035e+01   9.3493431465054497e-02   7.2044601527434982e-01
   8.0777472107017559e+01   7.5488426635847844e-02   5.7842767287300334e-01
   9.3744333162063725e+01   6.2674440382293584e-02   4.7389527126402314e-01
   1.1083771921146699e+02   5.1684746998972284e-02   3.9286739686708905e-01
   1.3268383473505730e+02   3.9942468624675084e-02   3.0246039013672588e-01
   1.5468031548972223e+02   3.2523070810002316e-02   2.5502915230851597e-01
   1.8499729727755485e+02   2.7385906222469465e-02   2.1844534896007733e-01
    }\LshapeADmixedBC
\pgfplotstableread{
   sqrtndofUN               Q1UN                     etaUN
   4.1231056256176606e+00   1.4974963855282994e+00   1.6850238315415169e+01
   1.0246950765959598e+01   7.6369435600623758e-01   6.3309035623531091e+00
   2.2293496809607955e+01   5.3105409906134970e-01   3.0056401214302677e+00
   4.6314144707637645e+01   4.0776027918087321e-01   2.0677478661364401e+00
   9.4323910012255112e+01   3.1951474867118396e-01   1.5927520202443850e+00
   1.9032866310674279e+02   2.5187714955573487e-01   1.2508790589741170e+00
    }\LshapeUNmixedBC

		\begin{loglogaxis}[legend pos=north east,legend cell align=left,
			legend style={fill=white},
			ymin=1e-2,ymax=1e2,
			xmin=1e0,xmax=1.6e3,
			ytick={1e-2,1e-1,1e0,1e1,1e2},
			xtick={1e0,1e1,1e2,1e3}]
			\pgfplotsset{
				cycle list={%
					{Black, mark=diamond*, mark size=2pt},
					{Black, mark=*, mark size=2pt},
					{Black, mark=x, mark size=3pt},
					{Black, mark=triangle, mark size=3pt},
					{Black, mark=square, mark size=3pt},
					{Black, mark=pentagon*, mark size=2pt},
					{Cyan, mark=*, mark size=1.5pt},
					{Orange, mark=square*, mark size=2.5pt},
					{Orange, mark=square*, mark size=2pt},
					{Orange, mark=square*, mark size=1.5pt},
				},
				legend style={
					at={(1,1)}, anchor=north east},
				font=\sffamily\scriptsize,
				xlabel=$\sqrt{\mathtt{ndof}}$,ylabel=
			}
			\addplot+ table[x=sqrtndofAD,y=Q1AD]{\LshapeADmixedBC};
			\addlegendentry{error, adapt, $V_h$}
			\addplot+ table[x=sqrtndofAD,y=etaAD]{\LshapeADmixedBC};
			\addlegendentry{$\bm\eta$, adapt, $V_h$}
			\addplot+ table[x=sqrtndofUN,y=Q1UN]{\LshapeUNmixedBC};
			\addlegendentry{error, unif, $V_h$}
			\addplot+ table[x=sqrtndofUN,y=etaUN]{\LshapeUNmixedBC};
			\addlegendentry{$\bm\eta$, unif, $V_h$}
%
			\addplot+ [color=Black,dashed,mark=none] coordinates{(4e-1,1e1) (1e3,4e-3)};
			\node(z) at  (axis cs:10,8e-2)
			[above] {$O (\sqrt{\mathtt{ndof}})$};
			\addplot+ [color=Black,dashed,mark=none] coordinates{(1,6) (1e3,6e-1)};
			\node(z) at  (axis cs:600,2e0)
			[below] {$O (\sqrt{\mathtt{ndof}}^{1/3})$};
		\end{loglogaxis}
	\end{tikzpicture}
\end{minipage}
\hfill
    \begin{minipage}{.43\textwidth}
	\includegraphics[width=.7\textwidth]{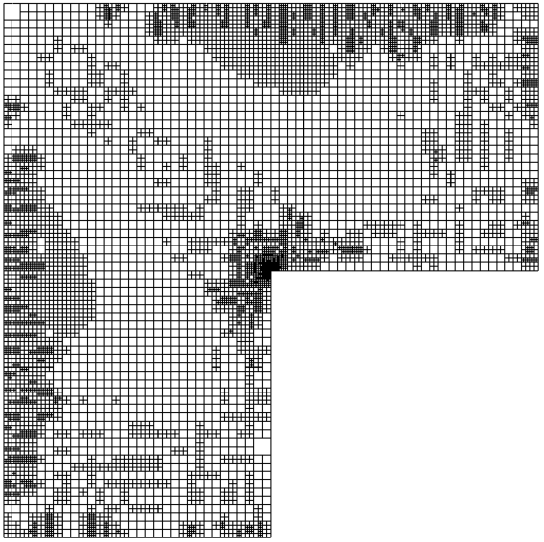}
	\end{minipage}
	\caption{%
	Convergence history of the error $\LLL u-u_h\RRR_h$
	and the estimator $\bm\eta$ for the L-shaped domain with 
	mixed boundary conditions from \S\ref{ss:results_Lmixed} (left).
	Adaptive mesh of this problem
	with $34\,224$ degrees of freedom, level $20$
 	    (right).
	\label{f:convLmix}
	}
\end{figure}

\appendix
\section{Properties of the Adini basis functions}
\label{app:adinibasis}

The Adini basis function $\varphi_{z,\alpha}\in\widehat V_h$
with respect to a regular vertex $z\in\mathcal V^{\mathrm{reg}}$
and a multiindex
$\alpha\in\{(0,0),(1,0),(0,1)\}$
satisfies
$$
 \partial^\beta \varphi_{z,\alpha} (\tilde z)
 =
 \delta_{z,\tilde z} \, \delta_{\alpha,\beta}
 \quad \text{for all }
 \tilde z \in\mathcal V^{\mathrm{reg}}
 \text{ and }
 \beta\in\{(0,0),(1,0),(0,1)\}
$$
with the Kronecker $\delta$.
This uniquely defines $\varphi_{z,\alpha}$
on regular partitions. On partitions with irregular vertices,
it uniquely defines $\varphi_{z,\alpha}$ on elements
$T$ with $\mathcal V(T)\subseteq\mathcal V^{\mathrm{reg}}$,
that is, on elements with only regular vertices.

In what follows we consider two rectangles $T$, $K$ sharing
an edge lying on the $y$ axis, with one vertex being $z=(0,0)$,
as shown in Figure~\ref{f:T1T4}.
The ratio of the $x$-widths of $T$, $K$ is denoted by
$$
 \rho=\operatorname{diam}_x (T) / \operatorname{diam}_x (K).
$$
The following results are formulated in the coordinates $(x,y)$
if $T,K$ are considered. If only $T$ is considered,
the usual local coordinates $\xi,\eta$ are employed.

\begin{figure}
 \begin{center}
 \begin{tikzpicture}[scale=1]
  \draw (0,0)--(3,0)--(3,1)--(0,1)--cycle;
  \draw (1,0)--(1,1);
  \fill (1,0) circle (.4ex);
  \node at (1,0) [below right]{$z=(0,0)$};
  \node at (.5,.5) {$K$};
  \node at (1.5,.5) {$T$};
\end{tikzpicture}
 \end{center}
 \caption{Vertex $z=(0,0)$ shared by $T$, $K$ with common edge lying on
          the $y$ axis.}
 \label{f:T1T4}
\end{figure}
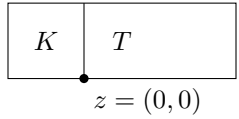

\begin{lemma}\label{l:nodalintegral}
In the configuration of Figure~\ref{f:T1T4},
$\varphi=\varphi_{z,(0,0)}$ satisfies
$$
 \mathrm{(a)} \int_T  q\, \partial^2_{xx} \varphi = 0,
 \quad
 \mathrm{(b)} \int_{T\cup K} x\,\partial^2_{xx} \varphi  = 0,
 \quad
 \mathrm{(c)} \int_{T\cup K} q\,\partial^2_{xy} \varphi  = 0
$$
for any integrable function $q=q(y)$ depending only on $y$.
\end{lemma}
\begin{proof}
 For the proof of (a), we use local coordinates $\xi,\eta$,
 write $\hat q(\eta)=q(y)$,
 and observe that $\varphi$ vanishes identically at $\eta=1$
 and $\partial^2_{xx} \varphi$ is bilinear, so that there exists a first-order
 polynomial $p_1=p_1(\xi)$ such that
 $\partial^2_{xx}\varphi=(1-\eta)p_1(\xi)$.
 Since $\partial_x\varphi$ vanishes at all vertices of $T$,
 considering $\eta=-1$ with the fundamental theorem of calculus
 shows that
 $$ \int_{-1}^1 \partial^2_{xx} p_1(\xi)\,d\xi =0.
 $$
 The original integral then reads
 $$
 \int_T  \hat q(\eta) \partial^2_{xx} \varphi
 =
 h_x h_y\int_{-1}^1 \hat q(\eta) (1-\eta)\,d\eta
  \int_{-1}^1\partial^2_{xx}p_1(\xi)\,d\xi
  =0.
 $$
 For the proof of (b), we use the symmetry
 $\varphi|_{K}(x,y) = \varphi|_{T}(-\rho x,y)$
 for $x\in K$.
 By the change of variables $\hat x= -\rho x$
 we have for the (undirected) volume integrals that
 $$
   \int_{K} x\partial^2_{xx} \varphi(x,y) \, dxdy
   =
   \int_{T} \frac{-\hat x}{\rho} \partial^2_{xx}\varphi (\hat x,y)
                   \lvert-\rho^{-1}\rvert  \, d\hat x dy
   =
   -\int_{T} \hat x \partial^2_{\hat x\hat x}\varphi (\hat x,y)   \, d\hat x dy
   .
 $$
 This implies (b). An analogous computation shows (c).
\end{proof}

\begin{lemma}\label{l:varphiz10}
In the configuration of Figure~\ref{f:T1T4},
$\varphi=\varphi_{z,(1,0)}$ satisfies
$$
 \mathrm{(a)} \int_T  x \,\partial^2_{xx} \varphi = 0,
 \quad
 \mathrm{(b)} \int_{T} q\,\partial^2_{xy} \varphi  = 0,
 \quad
 \mathrm{(c)} \int_{T\cup K} q\, \partial^2_{xx} \varphi  = 0
$$
for any integrable function $q=q(y)$ depending only on $y$.
\end{lemma}
\begin{proof}
 Since $\varphi$ vanishes on all sides apart from $\{\eta=-1\}$,
 it contains the linear factors
 $(\eta-1)(\xi+1)(\xi-1)$. Since $\varphi|_T\in\mathcal A$,
 for $\partial_x\varphi$ to vanish on the
 two endpoints of the face $\xi=1$, an additional factor $(\xi-1)$
 is necessary. The presence of the resulting factor $(\xi-1)^2$
 shows that $\varphi_x$ vanishes on the whole face $\xi=1$.
 Integration by parts then reads
 $$
  \int_{T}x\partial_{xx}^2\varphi
  =
  -
  \int_{T}\partial_{x}\varphi.
 $$
 This equals zero because $\varphi$ has zero boundary conditions,
 which proves (a).
 For the proof of (b),
 integration by parts with respect to $y$ shows
 $$
 h_y\int_{T} \hat q(\eta) \partial^2_{xy}\varphi
 =
   h_y
   \sum_{\sigma=\pm1}
   \int_{\{\eta=\sigma\}} \sigma \hat q(\sigma)\partial_x\varphi
  - \int_{T} \partial_y \hat q(\eta) \partial_x\varphi.
 $$
 Integration by parts with respect to $x$ shows
 that all these integrals vanish because $\varphi$ vanishes identically
 on the edges parallel to the $y$ axis. This shows (b).
 We note the symmetry
 $\varphi|_{K}(x,y) = -\rho^{-1}\varphi|_{T}(-\rho x,y)$.
 A computation analogous to that of the proof of (b) in
 Lemma~\ref{l:nodalintegral} thus proves (c).
\end{proof}

\begin{lemma}\label{l:varphiz01}
In the configuration of Figure~\ref{f:T1T4},
$\varphi=\varphi_{z,(0,1)}$ satisfies
$$
 \int_{T\cup K} q\,\partial^2_{xy} \varphi  = 0
 \quad\text{and}\quad
 \partial^2_{xx} \varphi=0
$$
for any integrable function $q=q(y)$ depending only on $y$.
\end{lemma}
\begin{proof}
 With the symmetry
 $\varphi|_{K}(x,y) = \varphi|_{T}(-\rho x,y)$,
 an argument similar to that of the proof of (b) in
 Lemma~\ref{l:nodalintegral} proves the first identity.
Since the bilinear function $\partial^2_{xx}\varphi$ vanishes on the
faces parallel to the $x$ axis, we have $\partial^2_{xx}\varphi=0$,
which is the second asserted identity.
\end{proof}

\begin{lemma}\label{l:d2orthp1}
Let $z$ be a regular interior vertex whose patch $\omega_z$
does not contain irregular vertices.
Any
$$
  \varphi \in \{\varphi_{z,(0,0)},\varphi_{z,(1,0)},\varphi_{z,(0,1)} \}
$$
out of the three global Adini basis functions related to $z$
satisfies
$$
 \int_{\omega_z} p\,\partial^2_{jk} \varphi \,dx = 0
 \quad\text{for any }p\in P_1
 \text{ and any pair } (j,k)\in\{1,2\}^2 .
$$
\end{lemma}
\begin{proof}
The assumptions on $\omega_z$ imply that the patch is a rectangle
and is formed by four rectangular elements.
The assertion thus follows from carefully combining the foregoing three lemmas
with suitable changes of coordinates.
\end{proof}

\begin{lemma}\label{l:dxxdyyorth_globalbf}
In the configuration of Figure~\ref{f:hanging},
the global Adini basis function $\varphi =\varphi_{z',(0,0)}$
related to the point value at $z'$ 
(with zero assignment for the normal derivatives at hanging nodes)
satisfies
$$
 \int_{\Omega} p\,\partial^2_{jj} \varphi \,dx = 0
 \quad\text{for any }p\in P_1
 \text{ and any } j\in\{1,2\}^2 .
$$
\end{lemma}
\begin{proof}
Let $\Phi$ denote the Adini basis function related to point evaluation
at $z'$ with respect to 
a coarser triangulation resulting from unifying the four smaller
rectangles in Figure~\ref{f:hanging}.
The function $\varphi$ can be written as
$$
 \varphi
  =\Phi 
    +\sum_{|\alpha|\leq 1} c_\alpha \varphi_{\hat z,\alpha}
    +\tilde c\psi_{\tilde z,(0,1)}
    +\check c\psi_{\check z,(1,0)}
$$
for suitable coefficients $c_\alpha$, $\tilde c$, $\check c$.
with $\tilde z$, $\hat z$, $\check z$ as in Figure~\ref{f:hanging}.
Lemma~\ref{l:d2orthp1}
and Lemma~\ref{l:varphiz10}--\ref{l:varphiz01}
(with changes of coordinates) can be applied and yield the assertion.
\end{proof}

\begin{lemma}\label{l:nonzeroint}
In the configuration of Figure~\ref{f:T1T4},
$\varphi=\varphi_{z,(0,1)}$ satisfies
$$
 \int_{T\cup K} D^2 \varphi = h_T \begin{bmatrix}
                               0 & 0\\0 & \gamma
                              \end{bmatrix}
\quad\text{with }\gamma\approx 1.
$$
\end{lemma}
\begin{proof}
 Lemma~\ref{l:varphiz01} shows that
 $\int_{T\cup K}\partial^2_{jk}\varphi=0$
 if $\min\{j,k\}\leq1$.
 From the symmetry
 $\varphi|_K(x,y) = \varphi|_T(-\rho x,y)$
 and change of variables we further obtain
 $$
 \int_K \partial_{yy}^2\varphi
 =
 \rho^{-1} \int_T \partial_{yy}^2\varphi
 \qquad\text{and therefore}\quad
 \int_{T\cup K} \partial_{yy}^2\varphi
 = (1+\rho^{-1})
   \int_T \partial_{yy}^2\varphi.
 $$
Since $\varphi$ vanishes on all edges of $T$ apart from
$\{\xi=-1\}$, an argument analogous to that of the proof of
Lemma~\ref{l:varphiz10} shows that
$\varphi = c (\eta-1)^2(\eta+1)(\xi-1)$ with some $c\approx h_y$.
Then, obviously, the integral of
$\partial_{yy}^2\varphi=h_y^{-2}c(\xi-1)(6\eta-2)$
over $T$ is nonzero and scales like $h_T$.
\end{proof}

\section{Bilinear interpolation with hanging-node constraint}
\label{app:Q1int}
Given a piecewise polynomial function $w$ that is continuous
in the regular vertices of $\mathcal T$,
its globally continuous and piecewise bilinear interpolation
$Q w$ is defined in \eqref{e:Qdef}.

\begin{lemma}[stability and approximation of bilinear interpolation]
\label{l:Q1stab}
Let the partition $\mathcal T$ satisfy the mesh condition
of Definition~\ref{d:meshcondition}
and let the function $w$ be globally continuous
and piecewise polynomial with respect to $\mathcal T$.
Then
$$
  h_T^{-2}\|w-Qw\|_T + h_T^{-1} \|\nabla (w-Qw)\|_T +\| D^2 Qw\|_T
  \lesssim
  \|D^2_h w\|_{\omega_T}
  \quad  \text{for any }T\in\mathcal T,
$$
with the element patch $\omega_T$.
The constant hidden in the notation $\lesssim$ depends on the polynomial
degree of $w$.
\end{lemma}
\begin{proof}
 If $T$ exclusively has regular vertices, $Q$ is the standard bilinear
 interpolation on $T$ and the result is obvious.
 Assume therefore that $T$ has an irregular vertex $z$.
 Then $z$ belongs to an edge with two neighbouring regular vertices
 one of them lying outside $T$. By the mesh condition,
 $T$ must possess at least two regular vertices, so that in total there are
 at least three 
 regular vertices inside $\overline{\omega}_T$
  that are not collinear.
 At these points $w=Qw$ holds.
 Since the value at irregular vertices is compatible
 with bilinear interpolation,
 $Q$ locally reproduces affine functions.
 The asserted result thus follows
 from a standard scaling argument and the
 finite number of possible local mesh configurations.
\end{proof}

\section{BFS averaging}\label{app:bfsaveraging}
It is well known from the Bogner--Fox--Schmid (BFS) finite element
\cite{Ciarlet1978}
that, on any rectangle, the 16 linear functionals
$$
  v \mapsto \partial^\alpha v(z)
  \quad\text{for any vertex }z\text{ of }T\text{ and any }
  \alpha\in\mathfrak B
$$
for the set
$\mathfrak B:=\{(0,0),(1,0),(0,1),(1,1)\}$ of multiindices
are linear independent over $Q_3$ (bicubic functions).
The corresponding dual BFS basis of $Q_3$ consists of the 16
functions $\psi_{z,\alpha}$ with
$$
  \partial^\beta\psi_{z,\alpha}(\tilde z)
  = \delta_{z,\tilde z}\delta_{\alpha,\beta}
  \quad \text{for all }
 z,\tilde z \in\mathcal V(T)
 \text{ and }
 \alpha,\beta\in\mathfrak B.
$$
As a basic averaging operator, introduce $\mathcal M$
and $\mathcal M_0$ mapping
a piecewise smooth function
$v$ to a piecewise $Q_3$ function
by assigning the mean of the above local functionals
at \emph{all vertices} (resp.\ all interior vertices).
More precisely,
for every rectangle $T\in\mathcal T$ and any vertex
$z\in\mathcal V(T)$ of $T$,
they are defined by
 $$
  (\partial^\alpha \mathcal M v|_T)(z)
  :=
  \overline{\sum_{\substack{K\in\mathcal T: \\ z\in K}}}
            (\partial^\alpha v|_K)(z)
  \qquad\text{and}\qquad
  (\partial^\alpha \mathcal M_0 v|_T)(z)
  :=
  \begin{cases}
   \mathcal M(z) &\text{if }z\in\Omega \\
   0             &\text{if }z\in\partial\Omega
  \end{cases}
 $$
 (the $\Sigma$ with the bar represents the average).
 This function is $C^1$ continuous in all regular vertices,
 but it may be discontinuous at irregular vertices.
 From $\mathcal M v$ (resp.\ $\mathcal M_0 v$)
 we construct a globally $C^1$
 and piecewise bicubic (thus BFS) function $\mathcal Jv$
 (resp.\ $\mathcal J_0v$)
 by assigning the values of $\mathcal Mv$
 (resp.\ $\mathcal M_0v$) at regular vertices
 and by matching the values at irregular vertices
 by interpolation,
 more precisely
 \begin{align*}
  (\partial^\alpha\mathcal Jv)(z) =
  \begin{cases}
  (\partial^\alpha\mathcal Mv)(z)
  &\text{if }z\in\mathcal V^{\mathrm{reg}}\\
  (\partial^\alpha\mathcal Mv|_T)(z)
  &\text{if }z\in\mathcal V^{\mathrm{irr}}
            \text{ and } z\in T\setminus\mathcal V(T).
  \end{cases}
 \end{align*}
 The definition of $\mathcal J_0$ is analogous with $\mathcal M$
 replaced by $\mathcal M_0$ in the above formula.
In the general case that $v$ is piecewise $H^2$-regular,
we overload notation and extend $\mathcal J$ by defining
$\mathcal Jv:=\mathcal J\Pi_{Q_3}v$
for the $L^2$ projection $\Pi_{Q_3}$ to piecewise bicubic functions.
As in prior sections we denote by $[\cdot]_E$ the jump across an edge $E$.

\begin{lemma}\label{l:bfsaveraging}
 Let $\mathcal T$ be a partition
 satisfying the mesh condition of Definition~\ref{d:meshcondition}.
 Any piecewise $Q_3$ (bicubic) function $v\in L^2(\Omega)$
 satisfies
 $$
  \| v-\mathcal J_0 v\|^2
  \lesssim
  \sum_{E\in\mathcal E}
  \left(
    h_E \| [v]_E \|_E^2
    +
    h_E^{3} \| [\nabla v]_E \cdot n_E\|_E^2
  \right).
 $$
If $v\in H^2(\Omega)$, we have
$$
 \|h^{-2} (v-\mathcal J v)\|
 +
 \|h^{-1} \nabla (v-\mathcal J v)\|
 +
 \|D^2 \mathcal Jv\|
 \lesssim
 \|D^2 v\|.
$$
\end{lemma}
\begin{proof}
Let $v$ be piecewise bicubic.
Let $T\in\mathcal T$ and $z\in\mathcal V(T)$ be a vertex of $T$.
 Standard techniques
 \cite{BrennerScott2008,CarstensenGallistlHu2013}
 reveal for the basic averaging operator $\mathcal M_0$ that
 $$
 \sum_{\alpha\in \mathfrak B}
    h_T^{1+|\alpha|} |\partial^\alpha (v-\mathcal M_0 v)(z)|
    \lesssim
    \left(
    \sum_{E\in\mathcal E: z\in E}
    (h_T
    \| [v]_E\|_E^2
    +
    h_T^{3}
    \|[\nabla v]_E\cdot n_E \|_E^2
    )
    \right)^{1/2}.
 $$
 If $z$ is a regular vertex, the same estimate obviously holds
 for $\mathcal M_0$ replaced by $\mathcal J_0$.
 If $z$ is an irregular vertex and $K$ is the element with
 $z\in K\setminus \mathcal V(K)$,
 then $(\partial^\alpha \mathcal J_0v)(z)$ is defined by interpolation
 from information of $\mathcal M_0v$ in the neighbouring vertices
 $z_1,z_2$; and a scaling argument shows that
 $$
   \partial^\alpha  (v-\mathcal J_0 v)|_T(z)
   =
   (\partial^\alpha v|_T-\partial^\alpha\mathcal M_0 v|_K)(z)
   \lesssim
   \sum_{j=1}^2
   \sum_{\beta\in \mathfrak B} h_T^{|\beta|-|\alpha|}
                 \left|\partial^\beta (v-\mathcal M_0 v)|_{T_j}(z_j)\right|
  ,
 $$
 where $T_1$, $T_2$ are the two rectangles
 with $\{z,z_j\}\subseteq \mathcal V(T_j)$ (one of them being $T$).
 Since the expansion of $v-\mathcal J_0v$ on $T$ in terms of the BFS basis
 functions and the scaling of the latter read
 $$
   (v-\mathcal J_0v)|_T
   =
   \sum_{\alpha\in\mathfrak B}
   \sum_{z\in\mathcal V(T)}
    \partial^\alpha (v-\mathcal J_0v)(z) \psi_{\alpha,z}
   \qquad\text{and}\qquad
   \|\psi_{\alpha,z}\|_T \lesssim h_T^{1+|\alpha|},
 $$
 a direct computation with the triangle inequality and the above estimates
 at the vertices
 and the local equivalence $h_T\approx h_E$ reveal that
 \begin{equation*}
  \|v-\mathcal J_0 v\|_T
  \lesssim
  \left(
  \sum_{z\in\mathcal V(T)}
  \sum_{E\in\mathcal E: z\in E}
    (h_E \| [v]_E\|_E^2
    +
    h_E^{3} \|[\nabla v]_E\cdot n_E \|_E^2
    )
    \right)^{1/2}.
 \end{equation*}
 This and the finite overlap of the element patches proves the first stated
 estimate for $\mathcal J_0$.
 An analogous argument shows that the same upper bound is valid for
 $\| v-\mathcal J v\|^2$.
 The second stated estimate follows from combining this bound with
 local trace inequalities and standard estimates for the piecewise
 $L^2$ projection.
\end{proof}

We remark that
in the upper bound of $\| v-\mathcal J v\|^2$ mentioned in the 
proof of the foregoing lemma,
the boundary edges can be dropped,
which is, however, not made use of in this work.

\section{Adini quasi-interpolation}
\label{app:qiAdini}

We denote by $I_h$ the standard
Adini interpolation with zero boundary data
acting an a sufficiently smooth function 
$w\in C^1(\overline\Omega)$ as
$$
 I_h w|_T =
     \sum_{z\in\mathcal V(T)\cap\Omega}
     \sum_{|\alpha|\leq 1}
      \partial^\alpha w(z)\varphi_{z,\alpha}
      \quad\text{for any }T\in\mathcal T
$$
with $\varphi_{z,\alpha}$ defined in \S\ref{app:adinibasis}.
Let $v\in H^2_0(\Omega)$.
We define the Adini quasi-interpolation
$I_h \mathcal J v \in V_h$,
where $\mathcal J$
is the BFS averaging from \S\ref{app:bfsaveraging}.
With Lemma~\ref{l:bfsaveraging} and standard discrete estimates we obtain
\begin{equation}\label{e:adiniQI}
 \|h^{-2} (v- I_h \mathcal J v)\|
 +
 \|h^{-1} \nabla (v- I_h \mathcal J v)\|
 +
 \|D^2_h I_h \mathcal J v\|
 \lesssim
 \|D^2 v\|.
\end{equation}

\bibliographystyle{abbrv}
\bibliography{plate}

\end{document}